\documentclass[11pt,reqno]{amsproc}

\title{Transverse Instability of Stokes Waves at Finite Depth}

\author{Ryan P. Creedon}
\address{Division of Applied Mathematics, Brown University, Providence, RI 02912}
\email{ryan$\_$creedon@brown.edu}

\author{Huy Q. Nguyen}
\address{Department of Mathematics, University of Maryland, College Park, MD 20742}
\email{hnguye90@umd.edu}

\author{Walter A. Strauss}
\address{Department of Mathematics, Brown University, Providence, RI 02912}
\email{walter$\_$strauss@brown.edu}
\usepackage[margin=1in]{geometry}
\usepackage{ulem}
\usepackage{amsmath, amsthm, amssymb, mathrsfs}
\usepackage{times}
\usepackage{color}
\usepackage{hyperref}
\usepackage{multirow}
\usepackage{graphicx}
\usepackage{tikz}
\usetikzlibrary{decorations.pathmorphing}
\usepackage{stix}

\newcommand{\bq}{\begin{equation}}
\newcommand{\eq}{\end{equation}}
\newcommand{\bqa}{\begin{eqnarray*}}
\newcommand{\eqa}{\end{eqnarray*}}

\theoremstyle{plain}
\newtheorem{theo}{Theorem}[section]
\newtheorem{prop}[theo]{Proposition}
\newtheorem{lemm}[theo]{Lemma}
\newtheorem{coro}[theo]{Corollary}

\newtheorem{defi}[theo]{Definition}
\theoremstyle{definition}
\newtheorem{rema}[theo]{Remark} 

\DeclareMathOperator{\cnx}{div}

\DeclareMathOperator{\I}{I}

\DeclareSymbolFont{pletters}{OT1}{cmr}{m}{sl}
\DeclareMathSymbol{s}{\mathalpha}{pletters}{`s}

\def\sign{\mathrm{sign}}
\def\tt{\theta}
\def\eps{\varepsilon}
\def\na{\nabla}

\def\wh{\widehat}
\def\g{\gamma}
\def\lb{\lambda}
\def\mez{\frac{1}{2}}
\def\tdm{\frac{3}{2}}

\def\sF{\mathscr F} 
\def\sG{\mathscr G}
\def\Rr{\mathbb{R}}

\def\Zz{\mathbb{Z}}

\def\cS{\mathcal{S}}

\def\cG{\mathcal{G}}
\def\cJ{\mathcal{J}}
\def\cK{\mathcal{K}}

\def\cT{\mathcal{T}}
\def\cL{\mathcal{L}}
\def\L1{\mathcal{L}^{(1)}}
\def\L2{\mathcal{L}^{(2)}}
\def\L3{\mathcal{L}^{(3)}}

\def\cO{O}
\def\cH{\mathcal{H}}
\def\cV{\mathcal{V}}
\def\ld{\lambda}

\def\p{\partial}

\def\na{\nabla}

\def\ka{\kappa}

\def\ol{\overline}
\def\T{\mathbb{T}}
\def\Tt{\Theta}
\def\wt{\widetilde}

\def\ka{\kappa}

\def\ld{\lambda}

\numberwithin{equation}{section}

\begin{document}
\begin{abstract}
A Stokes wave is a traveling free-surface periodic water wave  that is constant in the direction transverse to the direction of propagation. 
In 1981 McLean discovered via numerical methods that Stokes waves are unstable with respect to transverse perturbations. 
In \cite{CreNguStr} for the case of infinite depth we proved rigorously that the spectrum of the water wave system linearized at small Stokes waves, with respect to transverse perturbations, contains unstable eigenvalues lying approximately on an ellipse.  In this paper we consider the case of finite depth and prove that the same  spectral instability result holds for all but finitely many values of the depth. The computations  are considerably more complicated in the finite depth case.  

\end{abstract}


\noindent

\maketitle

\section{Introduction} 
We consider classical  water waves that are irrotational, incompressible,  and inviscid.  
The water lies below an unknown free surface $S$. 
 Such waves have  been studied for over two centuries, notably by Stokes \cite{Stokes}.  
A {\it Stokes wave} is a two-dimensional steady wave traveling in a fixed horizontal direction at a fixed speed $c$. It has been known for a century that 
a curve of small-amplitude Stokes waves exists \cite{Nekrasov, Civita, Struik}.  
Several decades ago it was proven that the curve extends to large amplitudes as well \cite{Krasovskii, Keady, Toland, AFT, Plotnikov}. 

A brief history of longitudinal (non-transverse) instabilities was given in \cite{CreNguStr}, to which the reader is referred.  That literature left open the question of 
whether a small Stokes wave could be unstable when perturbed in both horizontal directions but keeping the longitudinal period unperturbed. 
 This  {\it transverse instability} problem was studied numerically first by Bryant \cite{Bryant} and then  
by much more detailed work of McLean {\it et al} \cite {MMMSY, McLeanDeep, McLeanFinite}. 
While these remarkable papers did detect transverse instabilities, a mathematical proof of these {\it three-dimensional} instabilities has been missing ever since.  In \cite{CreNguStr} we considered the case of infinite depth and proved  rigorously that the spectrum of the water wave system linearized at small Stokes waves contains unstable eigenvalues lying approximately on an ellipse.  The purpose of the present paper is to prove such a spectral instability result for the case of finite depth.

For background and references on longitudinal (modulational and high frequency) instabilities, 
we refer to \cite{CreNguStr}.  To bring the references up to date,  we mention some very recent work on 
longitudinal instability by Berti et al \cite{Berti4} and on near-extreme Stokes waves by Deconinck et al \cite{DDS24}.  
It is important to note that there are several other models of water waves for which the transverse instability has been studied rigorously, for instance for gravity-capillary waves \cite{RouTzv, HTW}.  

As in \cite{CreNguStr} we now specify the parameters of our problem. Let $x$ and $y$ denote the horizontal variables 
and $z$ the vertical one.  
The problem depends in a significant way on the {\it depth} $h$.  
Consider the curve of Stokes waves traveling in the $x$-direction and with a given period, say $2\pi$ without loss of generality.  This curve is parametrized by a small parameter $\eps$ which represents the wave  amplitude of the Stokes waves.  
Such a steady wave can be described in the moving $(x,z)$ plane 
(where $x-ct$ is replaced by $x$) 
by its free surface 
$S=\{(x,y,z)\ |\ z=\eta^*(x;\eps)\}$ and by its velocity potential $\psi^*(x;\eps)$ restricted to $S$.

We take the perturbation of $\eta^*$ to have the form $\overline{\eta}(x)e^{\lambda t + i\alpha y}$, where $\overline{\eta}$ has the {\it same period} $2\pi$ as the Stokes wave, $\lambda \in \mathbb{C}$ is the growth rate of the perturbation, and $\alpha\in \Rr$ is the transverse  wave number  of the perturbation. In other words, we consider  purely transverse perturbations,  leaving out any effect of modulational instabilities. The problem is to find at least one value of $\alpha$  that leads to instability, that is, $\text{Re}\lambda > 0$.   After linearizing the nonlinear water wave system about a Stokes wave, introducing a ``good-unknown,'' and performing a conformal mapping  change of variables, we find that the exponents $\lambda$ are eigenvalues of a linear Hamiltonian operator $\mathcal{L}_{h,\varepsilon,\beta}$, where $\beta = \alpha^2$. 
Of course, the operator depends in a significant way on the  depth $h$.  

For $\varepsilon=0$ and depth $h$, the eigenvalues of the linearized operator are
\bq 
\ld^0_\pm(k, \beta)=i\left[\tanh^{1/2}(h)k\pm (k^2+\beta)^{1/4}\tanh^{1/2}\left(h\left(k^2+\beta \right)^{1/2} \right)\right],\quad k\in \Zz.
\eq
It is clear that the limit as $h\to\infty$ is simpler.  
Motivated by \cite{MMMSY}, we determine a {\it resonant transverse wave number} $\alpha_*$ so that the unperturbed operator $\cL_{h, 0, \beta_*}$ with $\beta_* = \alpha_*^2$ has an imaginary double eigenvalue $ i\sigma$. 
 This eigenvalue corresponds to the lowest possible resonance ($m=1$ in \eqref{resonancecond}) that generates a Type II transverse instability according to McLean \cite{MMMSY}, of which there are infinitely many higher-order resonances that potentially generate higher-order transverse instabilities. We expect however that higher-order transverse instabilities have smaller growth rates for small Stokes waves. More precisely, we conjecture that the $m^{th}$ resonance in \eqref{resonancecond} would lead to instabilities  at order $\eps^{2m+1}$.
 
 In order to capture the transverse instabilities we introduce a new small parameter $\delta$ for the perturbation of $\beta$ about $\beta_*$. Our main result is that {\it the perturbed operator $\cL_{h, \eps, \beta_*+\delta}$ has eigenvalues $\ld_\pm$ 
with non-zero real parts that bifurcate from $i\sigma$}, stated more precisely in the following theorem.  %one of the real parts being positive. 

%%%%%%%%%%%%%%%%
\begin{theo} \label{theo:main}
 Let the depth $h\in (0,\infty)$ be given.  
Except for a finite number $N$ of values of $h$, the following instability statement is true. 
There exist $\varepsilon_{\textrm{max}} > 0$ and $\delta_{\text{max}}>0$  such that  
for all $ \varepsilon \in (- \varepsilon_{\textrm{max}},  \varepsilon_{\textrm{max}})$ and $\delta\in(-\delta_{\text{max}},\delta_{\text{max}})$, 
the operator $\mathcal{L}_{h, \varepsilon,\beta_*+\delta}$ has a pair of eigenvalues
\begin{align}
\lambda_{\pm} &=i\Big(\sigma + \frac{1}{2}T(\varepsilon,\delta) \Big) \pm \frac12 \sqrt{\Delta(\varepsilon,\delta)}, \label{lambda_exact1}
\end{align}
where $T$ and $\Delta$ are real-valued, real-analytic functions such that $T(\varepsilon,\delta) = O\left(\delta \right)$ and $\Delta(\varepsilon,\delta) = O\left(\delta^2\right)$ as $(\varepsilon,\delta) \rightarrow (0,0)$. Furthermore, there exist $\ka_0 \in \mathbb{R}$ and $\ka_1>0$ such that for 
\begin{align}
\delta = \delta(\varepsilon,\theta) = \kappa_0\varepsilon^2 + \theta\varepsilon^3 \quad \textrm{with} \quad |\theta|<\kappa_1,
\end{align}
we have $\Delta(\varepsilon,\delta(\varepsilon,\theta))>0$ for sufficiently small $\varepsilon$.
Thus the eigenvalue $\lambda_+$  has positive real part provided $\delta = \delta(\varepsilon,\theta)$ with $|\theta|<\kappa_1$ and $\varepsilon$ is sufficiently small. Moreover, $\text{Re}\lambda_+ = O\left(\varepsilon^3\right)$ as $\varepsilon \rightarrow 0$ for each $\theta$.
This means that there exist transverse perturbations of the given Stokes wave whose amplitudes 
grow temporally like $e^{t\,\text{Re}\ld_+ }$.
\end{theo}
 \begin{rema} 
 In the corresponding longitudinal result in \cite{Berti4}, there could be infinitely many exceptional values of the depth, whereas  in the transverse case Theorem \ref{theo:main} asserts that only finitely many such values can exist.  In fact, numerical computation shows that $N=1$, that is, there is only one such exceptional value of the depth. See Figure 3 in Section \ref{Sec:proof}. It should also be noted that small Stokes waves at the  exceptional depth(s) could still be susceptible to higher-order transverse instabilities.  
\end{rema}  
Fixing $\varepsilon$, 
substituting $\delta = \delta(\varepsilon,\theta)$ into \eqref{lambda_exact1} and dropping terms of $O\left(\varepsilon^4\right)$ and smaller, we obtain an asymptotic expansion of the unstable eigenvalues. As $\theta$ varies, we find that the eigenvalues lie approximately on an ellipse (an `isola') that is centered on the imaginary axis.  Figure 1 illustrates this ellipse for three different depths.  
As $\varepsilon$ varies, the center of the ellipse drifts from a double eigenvalue on the imaginary axis like $O\left(\varepsilon^2\right)$, while its semi-major and semi-minor axes scale like $O\left(\varepsilon^3\right)$.  We refer to Corollary \ref{cor:ellipse} for the precise statement. 
\begin{rema}
Shortly after the appearance of the preprint version of this paper, the preprint  \cite{JRY} of Jiao-Rodrigues-Yang proved the oblique instability of  small Stokes waves when additional localized perturbations  in the  longitudinal direction are included.  They thereby obtained an $\eps^2$-order instability, which is expected from the dominant modulational instability \cite{NguyenStrauss, Berti1}. 
\end{rema}

\begin{rema} 
Figure 1 compares our transverse instability isolae with the numerically computed eigenvalues  for a Stokes wave with amplitude $\varepsilon = 0.01$ at the three depths $h = 1$, $3/2$, and $2$. The eigenvalues are computed using the Floquet-Fourier-Hill method applied to the Ablowitz-Fokas-Musslimani formulation of the finite-depth water wave problem. Full details of this method are presented in \cite{DecOli15}. The distance between our transverse instability isolae and the numerically computed eigenvalues is $O\left(\varepsilon^4\right)$, as expected, since we do not calculate beyond $O\left(\varepsilon^3\right)$ to construct our isolae. The strong agreement between our isolae and the numerically computed eigenvalues to $O\left(\varepsilon^3\right)$ provides confidence both in our analytical results and in earlier numerical investigations of the transverse instability spectrum.  
Even better agreement could be found by retaining higher-order corrections of the unstable eigenvalues in a manner similar to \cite{CreDecTri}.
\end{rema}

\begin{figure}[t]
\includegraphics[width=14cm]{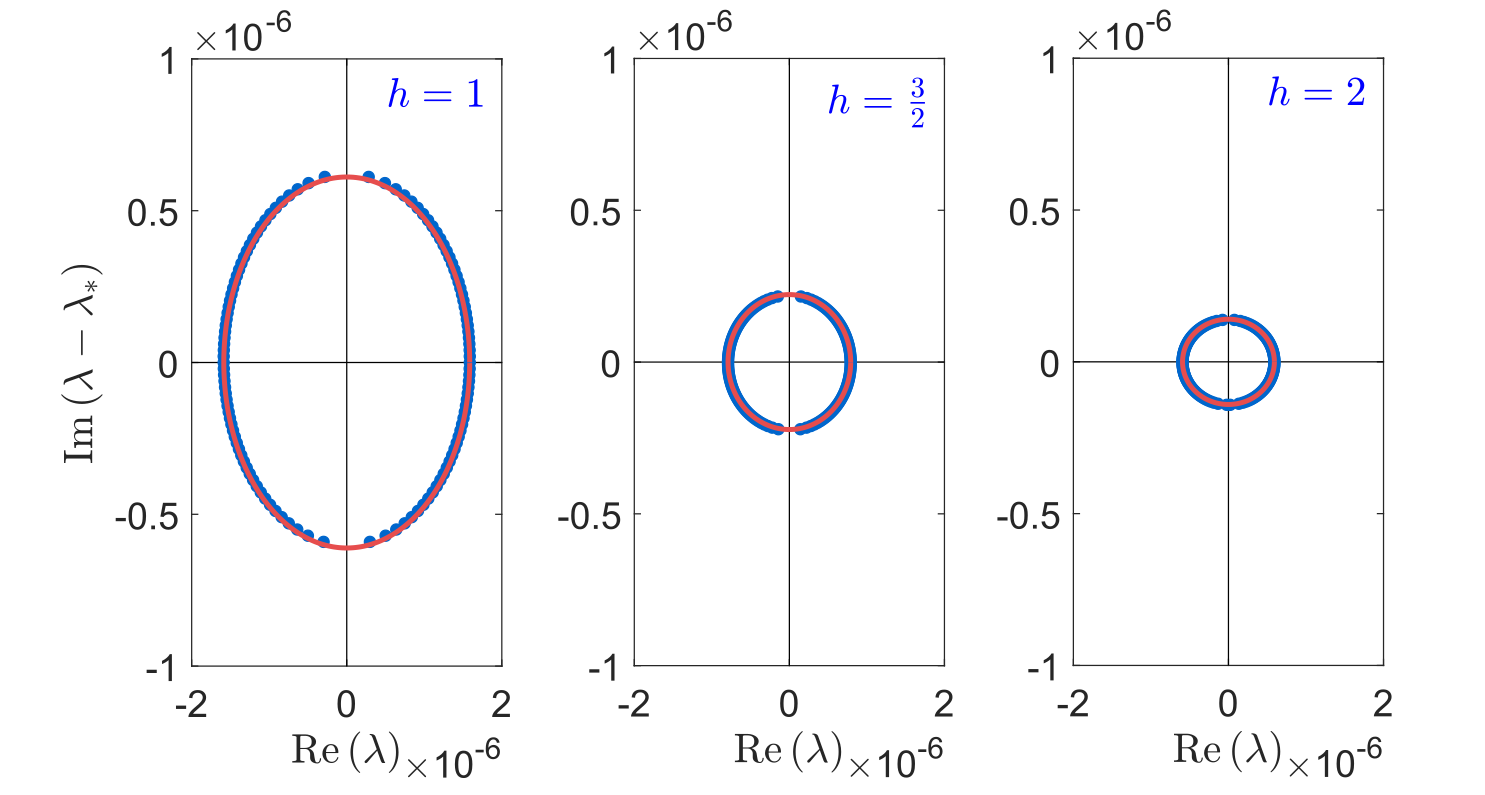}
\label{fig1}
\caption{A comparison of the transverse instability isola obtained in this work (orange curves) and numerical computations of the unstable eigenvalues of the transverse instability (blue dots) for a Stokes wave with amplitude $\varepsilon = 0.01$ in water of depth $h = 1$ (left), $h = 3/2$ (middle), and $h = 2$ (right). The center of the isola  is subtracted from its imaginary component to show a sense of scale. In each plot the distance between the orange curves and blue dots is $O\left(\varepsilon^4\right)$.  }
\end{figure}

%The isola of unstable eigenvalues found above is reminiscent of the high-frequency isolas that appear in the longitudinal stability spectrum. It is therefore natural to compare the growth rates of the transverse instability obtained in this work to the known growth rates of the longitudinal instabilities of Stokes waves, including the high-frequency and Benjamin-Feir instabilities. In the infinite depth longitudinal case, the largest high-frequency isola has semi-major and semi-minor axes that scale like $O\left(\varepsilon^4\right)$ \cite{CreDecTri}.  Thus our transverse instability grows at a {\it faster} rate $O(\eps^3)$ for sufficiently small amplitude waves.  
%On the other hand, our instability grows {\it slower} than the Benjamin-Feir instability rate, which is
%$O\left(\varepsilon^2\right)$ in both finite and infinite depth 
%\cite{BM, NguyenStrauss, Berti1,Berti2, HurYang, CreDec}. Moreover, our instability grows {\it slower} than the largest high-frequency instability in finite depth, which grows like $\cO(\varepsilon^2)$ \cite{CreDecTri, HurYang}.

We now turn to the  main ideas in the proof of Theorem \ref{theo:main}.  
It  begins by finding an  expression for 
$\cL_{h,\eps,\beta}$ by a method analogous to that in \cite{NguyenStrauss} and \cite{CreNguStr}. %However, for the present three-dimensional instability problem, $\cL_{h,\eps,\beta}$ involves a genuine pseudo-differential operator  as opposed to the Fourier multiplier  in the two-dimensional problem considered in \cite{NguyenStrauss, Berti1}.  
 It continues by following the method of \cite{Berti1} that uses a Kato similarity transformation 
to reduce the relevant spectral data of $\cL_{h,\eps,\beta}$ to a $2 \times 2$ matrix $\textrm{L}_{h,\eps, \delta}$ with the property that $i\textrm{L}_{h,\eps, \delta}$ is real and skew-adjoint.  
We prove that the entries of this matrix are real analytic functions of $\varepsilon$ 
and $\delta$ and we obtain convenient functional expressions for its eigenvalues, resulting in \eqref{lambda_exact1}. 
Compared to the infinite depth case  \cite{CreNguStr}, these expressions are much more complicated.  
In order to conclude that $\Delta(\varepsilon,\beta)>0$ for $\delta = \delta(\varepsilon,\theta)$ and sufficiently small $\varepsilon$, 
we must expand the entries of the matrix in a power series up to third order in the pair 
$(\varepsilon,\delta)$.  Second order is insufficient.  
The extremely arduous calculations require us to take advantage of Mathematica. They can be found in the companion Mathematica file {\it CompanionToTransverseInstabilities.nb}.

In Section 2 we introduce the Stokes waves and proceed with the  linearization and the flattening by means of a conformal mapping from a  strip   $\Omega_\eps=\Rr\times (-h_\eps, 0)$  to the two-dimensional fluid domain of the Stokes wave.  The main result that is required here is Theorem \ref{theo:flattenG}, which is devoted to the three-dimensional Dirichlet-Neumann operator $\cG_{\eps,\beta}$ that occurs in the linearized operator subject to the two-dimensional conformal mapping, and the proof of its analyticity in $\eps$ and $\delta$. The analyticity is proven  by an  approach inspired by Groves \cite{Groves} and is significantly shorter than our direct approach in \cite{CreNguStr}.

Theorem \ref{theo:main} is then reduced to studying the eigenvalues of the linearized 
operator $\cL_{h,\eps,\beta}$, which has a Hamiltonian form and is reversible.  
Section 3 is devoted to a discussion of the first resonance on the imaginary axis when $\eps=0$, the introduction of Kato's perturbed basis, 
and the reduction to the study of the eigenvalues of a reduced $2\times 2$ matrix $\textrm{L}_{h,\eps,\delta}$.  
In Section 4 we perform  expansions of $\cL_{h,\eps,\beta_*+\delta}$ out to third order in both $\eps$ and $\delta$, which are considerably more tedious than the corresponding expansions in \cite{CreNguStr,NguyenStrauss, Berti1}. 
We use the expansions of $\cL_{h, \eps,\beta_*+\delta}$ to compute the expansions of the Kato basis vectors and of the matrix $\textrm{L}_{h,\eps,\delta}$. 
Finally in Section 5 we analyze the leading terms in the characteristic discriminant of $\textrm{L}_{h,\eps,\delta}$.  An important step is to prove that a key coefficient, which we call $b_{3,0}$, does not vanish. 
 It is here that we must exclude the exceptional depths.  By determining the precise asymptotics of $b_{3, 0}$ as $h\to 0^+$ and $h\to \infty$, we are able to prove that there are at least one but at most only finitely many such exceptional depths. 
This concludes the proof of Theorem \ref{theo:main}.  The full equation that defines the isola 
of eigenvalues up to $O(\eps^4)$  is given in Corollary \ref{cor:ellipse}. 
%%%%%%%%  SECTION 2  %%%%%%%%%%%%%%%%%%%%%%%%
\section{Transverse perturbations of Stokes waves}
\subsection{Stokes waves}
We consider the three-dimensional  fluid domain 
\bq
D(t)=\{(x, y,z)\in \Rr^3\ :\ -h<z<\eta(x,y, t)\}
\eq
with free surface $S(t)=\{(x ,y, \eta(x,y, t)): (x,y)\in \Rr^2 \}$ over a constant depth $h > 0$.  Assuming that the fluid is incompressible,  inviscid  and irrotational, the velocity field $\bf u$ admits a harmonic 
potential $\phi(x, y,  z, t): D(t)\to \Rr$. That is, $\bf u = \nabla\phi$.  Then $\phi$ and $\eta$ satisfy the water wave system
\bq\label{ww:0}
\begin{cases}
 \Delta_{x, y,z} \phi =0\quad \text{ in } D(t), \\
\p_t\phi + \tfrac12 |\na_{x, y,z}\phi|^2 =- g\eta +P\quad \text{ on } S(t), \\
 \p_t\eta+\p_x\phi\p_x\eta=\p_y\phi\quad  \text{ on } S(t), \\
\partial_z\phi(x,y,-h) = 0,
\end{cases}
\eq
where $P$ is the Bernoulli constant and $g>0$ is the constant acceleration due to gravity. The second equation is Bernoulli's,  which follows from the 
pressure being constant along the free surface; the third equation expresses the kinematic boundary condition that particles on the surface remain there; 
the last condition asserts that there is no flux of water out of the bottom boundary $\{z = -h\}$.  Without loss of generality we will take $P=0$. 

 In order to reduce the system to the free surface $S$, we introduce the Dirichlet-Neumann operator $G(\eta)$ 
 associated to $D$, namely, 
\bq\label{def:Gh}
(G(\eta)f)(x,y) := \p_z\vartheta(x,y, \eta(x,y))-\na_{x,y}\vartheta(x,y, \eta(x,y)) \cdot \na_{x,y}\eta(x,y),
\eq
where $\vartheta(x, y,z)$ solves the elliptic problem 
\bq\label{elliptic:G}
\begin{cases}
\Delta_{x, y,z}\vartheta=0\quad\text{in}~D,\\
\vartheta\vert_S=f(x,y), \\
\partial_z\vartheta(x,y,-h) = 0.
\end{cases}
\eq
 We define $\psi$, not as the stream function, but as the trace of the velocity potential on the free surface, $\psi(x,y, t)=\phi(x, y, \eta(x,y, t), t)$. Then, in the moving frame with speed $(c, 0)\in\Rr^2$, the gravity water wave system 
written in the Zakharov-Craig-Sulem formulation \cite{Zak, CraSul} is 
 \bq\label{ww}
 \begin{cases}
 \p_t \eta=c\p_x\eta+ G(\eta)\psi=:F_1(\eta, \psi,  c),\\
 \p_t\psi=c\p_x\psi-\mez|\na_{x,y}\psi|^2 
 + \mez\frac{\big( G(\eta)\psi+\na_{x,y}\psi\cdot\na_{x,y}\eta\big)^2}{1+|\na_{x,y}\eta|^2}-g\eta=:F_2(\eta, \psi,  c),
 \end{cases}
 \eq
% By the change of variables 
% \[
%c\to \frac{c}{\sqrt{g}},\quad \eta(x, y, t)\to \eta(x, y, \sqrt{g}t),\quad  \psi(x, y, t)\to \sqrt{g}\psi(x, y, \sqrt{g}t),
%  \]
%  we henceforth assume without loss of generality that $g=1$.
For any $a>0$, if we make the change of variables 
\[
\eta(x, y, t)=\frac{1}{a} \tilde{\eta}(ax, ay, \sqrt{g}t),\quad \psi(x, y)=\frac{\sqrt{g}}{a\sqrt{a}} \tilde{\psi}(ax, ay, \sqrt{g}t),
\]
 then $(\tilde{\eta}, \tilde{\psi})$ is a solution of \eqref{ww}  with  $(g, c, h)$ replaced by $(1, \sqrt{\frac{a}{g}} c, a h)$. In particular, we can assume without loss of generality that $g=1$ in \eqref{ww}.
  
% Assuming waves travel in the $x$-direction, we have ${\bf c} = (c,0)$, and by the change of variables 
%\bq\label{rescale:g}
% (\psi, c)\to (\sqrt{g} \psi, \sqrt{g}c),
% \eq
% we can consider $g = 1$ without loss of generality, whereby 
%  \bq\label{ww:c}
% \begin{cases}
% \p_t \eta=c\eta_x+ G(\eta)\psi,\\
% \p_t\psi=c\psi_x-\mez|\na_{x,y}\psi|^2 
% + \mez\frac{\big( G(\eta)\psi+\na_{x,y}\psi\cdot\na_{x,y}\eta\big)^2}{1+|\na_{x,y}\eta|^2}-\eta.
% \end{cases}
% \eq
% By a steady wave we mean that $\eta$ is a function of $x-ct$ and $\phi$ a function of $(x-ct,y)$. 
 By a {\it Stokes wave} we mean a periodic steady solution of \eqref{ww} that is independent of $y$.  %, that is,  a pair $(\eta(x), \psi(x))$ satisfying 
%In the moving frame 
%with constant speed $c=(c_1, c_2)\in \Rr^2$, the  steady water wave system is given by
% \bq\label{sys:Stokes}
% \begin{cases}
%F_1(\eta, \psi, c):=c\eta_x+ G(\eta)\psi=0,\\
% F_2(\eta, \psi, c):=c\psi_x-\mez|\na_{x, y}\psi|^2+\mez\frac{\big( G(\eta)\psi+\na_{x, y}\psi\cdot\na_{x, y}\eta\big)^2}{1+|\na_{x, y}\eta|^2}-\eta=0.
% \end{cases}
% \eq
 %%%%%%%%%%%%%%%%%%%%%%
 %Consider now a Stokes wave  $(\eta^*(x), \psi^*(x))$ with  speed $c^*$. 
In view of the aforementioned scaling, without loss of generality, we henceforth consider $2\pi$-periodic Stokes waves.   As has been known for over a century,  
  there exists a curve of small Stokes waves parametrized analytically by the small amplitude $\eps$.  
  \begin{theo}     \label{theo:Stokes} 
 For any physical depth $h>0$ and any index $s>5/2$, there exist $\eps_{St}(s)>0$ depending on $h$ and a unique family of Stokes waves
 \[
 \big (\eta^*(x; \eps), \psi^*(x; \eps), c^*(\eps)\big)\in H^s(\T)\times H^s(\T)\times \Rr
 \]
 parametrized by $|\eps|<\eps_{St}(s)$, such that 
 \begin{itemize}
 \item(i) the mapping $(-\eps_{St}(s), \eps_{St}(s))\ni \eps\to  \big (\eta^*(\cdot; \eps), \psi^*(\cdot;\eps ), c^*(\eps)\big)\in H^s(\T)\times H^s(\T)\times \Rr$ is analytic;
 \item(ii) $\eta^*(\cdot;\eps )$ is even, and $\psi^*(\cdot;\eps  )$ is odd. 
 \end{itemize}
 \end{theo}
 \begin{proof} See Theorem 1.3 in \cite{Berti-DN} for instance.  \end{proof} 
We will need the following third-order expansions in $\eps$ of the above Stokes waves.   
 \begin{prop}    \label{eta expansion}%[\protect{Theorem 1.3, \cite{Berti-DN}}]
 The unique analytic family of Stokes waves given by Theorem \ref{theo:Stokes} has the expansions 
\begin{align} \label{expand:star}
c^*(\eps)&=c_0+ c_2 \eps^2+O(\eps^4), \\
\eta^*(x;\eps)&= \eps \cos x+ \eps^2\Big[\eta_{2,0} + \eta_{2,2}\cos(2x) \Big] + \eps^3\Big[\eta_{3,1}\cos x + \eta_{3,3}\cos(3x) \Big] + O(\eps^4),\\
\psi^*(x;\eps)&= \eps c_0^{-1}\sin x+\eps^2\psi_{2,2}\sin (2 x)+ \eps^3\Big[\psi_{3,1}\sin x+\psi_{3,3}\sin(3x)\Big] + O(\eps^4),
\end{align}
where
\allowdisplaybreaks
\begin{align}
c_0 = \sqrt{\tanh (h)}, \qquad &
c_2 =\frac{-12c_0^{12} + 13c_0^8 - 12c_0^4 + 9}{16c_0^7}, \\
\eta_{2,0} = \frac{c_0^4 - 1}{4c_0^2}, \quad
\eta_{2,2} = \frac{-c_0^4+3}{4c_0^6}, \quad &
\eta_{3,1} = \frac{-2c_0^{12}+3c_0^8+3}{16c_0^8(1+c_0^2)}, \quad
\eta_{3,3} = \frac{-3c_0^{12}+9c_0^8-9c_0^4+27}{64c_0^{12}}, \\
\psi_{2,2} = \frac{c_0^8+3}{8c_0^7}, \quad
\psi_{3,1} = &\frac{2c_0^{12}-8c_0^8-3}{16c_0^{7}(1+c_0^2)}, \quad
\psi_{3,3} = \frac{-9c_0^{12}+19c_0^8+5c_0^4+9}{64c_0^{13}}. 
\end{align}
\end{prop} \begin{proof}  See (1.9) in   \cite{Berti3}.   \end{proof} 
%From now on we consider the depth $h$ as fixed and take $g=1$.
%%%%%%%%%%%%%%%%%%%%%%%%%%
\subsection{Transverse perturbations, linearization, and transformations}
Let us fix a Stokes wave $(\eta^*, \psi^*, c^*)$ with amplitude $\eps$ and  perturb it by a two-dimensional perturbation:
\bq\label{perturbation:etapsi}
\eta(x, y)=\eta^*(x)+\nu \ol{\eta}(x, y),\quad \psi(x, y)=\psi^*(x)+\nu \ol{\psi}(x, y),\quad |\nu|\ll 1.
\eq
To prepare for the linearization of the dynamics of the full water wave equations,  
%using \eqref{perturbation:etapsi},  
we  recall the  shape-derivative formula for the derivative of $G(\eta)\psi$ with respect to $\eta$.

\begin{theo}    \label{theo:shapederi}
Let $\eta: \T^2\to \Rr$ and $\psi: \T^2\to \Rr$. The shape-derivative of $G(\eta)\psi$ with respect to $\eta$ is denoted by 
\bq
G'(\eta)\overline{\eta}\psi=\lim_{h\to 0}\frac{1}{h}\left(G(\eta+h\overline{\eta})\psi-G(\eta)\psi\right).
\eq
We have
\bq\label{shapederi:form}
 G'(\eta)\ol{\eta}\psi=-G(\eta)(\ol\eta B(\eta)\psi)-\cnx(\ol\eta V(\eta)\psi),
\eq
where
\bq\label{def:BV:0}
B(\eta)\psi:=\frac{G(\eta)\psi+\na \eta\cdot \na \psi }{1+|\na \eta|^2},\quad V(\eta)\psi:=\na \psi-\na \eta B(\eta)\psi.
\eq
\end{theo}
\begin{proof} See Theorem 3.20 in \cite{Lannes} and Proposition 2.11 in \cite{ABZ1}.    
\end{proof} 
%%%%%%%%%%
Using the formulas \eqref{shapederi:form}, \eqref{def:BV:0}, and the fact that the Stokes wave  is independent of the transverse variable $y$, we  {\it linearize} \eqref{ww} around it to obtain 
\begin{align}
&\p_t \ol\eta=\frac{\delta F_1(\eta^*, \psi^*, c^*)}{\delta(\eta, \psi)}(\ol\eta, \ol \psi)=\p_x\big((c^*-V^*)\ol\eta\big)+G(\eta^*) (\ol\psi-{B^*} \ol\eta),\label{lin:eta}\\
&\p_t\ol \psi=\frac{\delta F_2(\eta^*, \psi^*, c^*, P^*)}{\delta(\eta, \psi)}(\ol\eta, \ol \psi)=(c^*-{V^*})\p_x\ol\psi +{B^*}G(\eta^*)(\ol\psi-{B^*} \ol\eta)-{B^*}\p_x{V^*}\ol\eta-g\ol\eta,\label{lin:psi}
\end{align}
where
\bq\label{def:B*V*}
B^*:=B(\eta^*)\psi^*=\frac{G(\eta^*)\psi^*+\p_x\psi^*\p_x\eta^*}{1+|\p_x\eta^*|^2},\qquad V^*:=V(\eta^*)\psi^*=\p_x\psi^*-B^*\p_x\eta^*.
\eq 
A similar derivation was done in Lemma 3.2, \cite{NguyenStrauss} for perturbations $\ol{\eta}$ and $\ol{\psi}$ depending only on $x$. 
%-----------------------------------------------
%{\it Proof of \eqref{lin:eta}:
%
%From \eqref{perturbation:etapsi} we have
%\bq
%G(\eta)\psi=G(\eta^*+\nu \ol{\eta})(\psi^*+\nu \ol{\psi})=G(\eta^*+\nu \ol{\eta})\psi^*+\nu G(\eta^*+\nu \ol{\eta})\ol{\psi}.
%\eq
%Hence the linearization of $G(\eta)\psi$ about $(\eta^*, \psi^*)$ is 
%\bq
%G'(\eta^*)\ol{\eta}\psi^*+ G(\eta^*)\ol{\psi},
%\eq
%Applying  Theorem \ref{theo:shapederi} we obtain
%\bq
%G'(\eta^*)\ol{\eta}\psi^*=-G(\eta^*)(\ol{\eta}B)-\cnx(\ol\eta V),
%\eq
%where 
%\[
%B=\frac{G(\eta^*)\psi^*+\na \eta^*\cdot \na \psi^*}{1+|\na \eta^*|^2}=:B^*,\quad V=\na \psi^*-B\na \psi^*=:(V^*, 0).
%\]
%Since $\eta^*$ and $\psi^*$ are functions of $x$ only, we have 
%\bq
%B=\frac{G(\eta^*)\psi^*+\p_x \eta^* \p_x \psi^*}{1+|\p_x \eta^*|^2},\quad V=(\p_x \psi^*-B\p_x \psi^*, 0),
%\eq
%whence
%\bq
%G'(\eta^*)\ol{\eta}\psi^*=-G(\eta^*)(\ol{\eta}B^*)-\p_x(\ol\eta V^*).
%\eq
%------------------------------------------------
%}
It is worthwhile noting that in fact  $B^*$ and $V^*$ are the vertical and horizontal 
components, respectively,  of the fluid velocity of the Stokes wave at the free surface $\{z=\eta^*(x)\}$. 
Where $G(\eta^*)$ acts on functions of $(x, y)$ we consider $\eta^*(x, y)\equiv \eta^*(x)$. 
We change variables to the so-called {\it Alinhac's good unknowns} 
\bq\label{def:v12}
v_1(x, y)=\ol\eta\quad\text{and}~ v_2(x, y)=\ol\psi-{B^*}\ol\eta
\eq
satisfying
 \begin{align}\label{eq:v1}
&\p_t v_1=\p_x\big((c^*-{V^*})v_1\big)+G(\eta^*)v_2,\\ \label{eq:v2}
&\p_t v_2=-\big(g+({V^*}-c^*)\p_x{B^*} \big)v_1+(c^*-{V^*})\p_xv_2.
\end{align}

Choosing a simple form for the transverse perturbations, we consider perturbations 
$\ol{\eta}$ and $\ol{\psi}$ that have wave number $\alpha \in \Rr\setminus\{0\}$ in the transverse variable $y$: 
\bq
\ol{\eta}(x, y)=\wt\eta(x)e^{i\alpha y},\quad \ol{\psi}(x, y)=\wt\psi(x)e^{i\alpha y}.  
\eq
Here $\wt\eta$ and $\wt \psi$ are $2\pi$-periodic functions, so the perturbations have the same longitudinal wave length as that of the Stokes wave.   Consequently, the good unknowns have the form
\bq
v_1(x, y)=\wt\eta(x)e^{i\alpha y}:=\wt{v_1}(x)e^{i\alpha y},\quad v_2(x, y)=(\wt\psi-{B^*}\wt\eta)e^{i\alpha y}:=\wt{v_2}(x)e^{i\alpha y}.
\eq
Then, in the linearized system \eqref{eq:v1}-\eqref{eq:v2}, the most difficult term to analyze is $G(\eta^*)(\wt{v_2}(x)e^{i\alpha y})$. To handle this term, we shall flatten the surface 
$\{z=\eta^*(x)\}$  using  a  Riemann mapping. 
For the sake of clarity, from now on we denote the independent variables in the physical space 
 by $(X,y,Z)$.   The {\it Riemann mapping} is as follows. 
%%%%%%%%%%%%%%%%%%
 \begin{theo}  %[\protect{Proposition 3.3, \cite{NguyenStrauss}}]  
 \label{prop:Riemann}
 There exists a  holomorphic bijection $X(x, z)+iZ(x, z)$ 
 from a strip $\Omega_\eps=\{(x, z)\in \Rr^2: -h_\eps<z<0\}$  
 onto  $\Omega=\{(X,Z)\in \Rr^2: -h<Z<\eta^*(X)\}$ with the following properties.  
 \begin{itemize}
 \item[(i)] 
 $X(x+2\pi, z)=2\pi +X(x, z)$ and $Z(x+2\pi, z)=Z(x, z)$ 
 for all $(x, z)\in \Omega_\eps$; 
 \newline $X$ is odd in $x$ and $Z$ is even in $x$;
 \item[(ii)] 
 $X+iZ$ maps  $\{(x, 0): x\in \Rr\}$ onto the surface $S=\{(x, \eta^*(x)): x\in \Rr\}$ and maps $\{(x, -h_\eps): x\in \Rr\}$ onto the bottom $\{(x, -h): x\in \Rr\}$.
 \item[(iii)] 
 Defining the ``Riemann stretch" as 
 \bq\label{def:zeta}
 \zeta(x)=X(x, 0),
 \eq 
we have the Fourier expansions
 \bq\label{z1}
X(x, z)=x-\frac{i}{2\pi}\sum_{k\in \Zz\setminus\{0\}}e^{ikx}\mathrm{sign}(k) 
\frac{ \cosh(|k|(z+h_\eps) }  { \sinh(|k|h_\eps) } 
\widehat{\eta^*\circ \zeta}(k)
\eq
\bq\label{z2}
Z(x,z) = z+ h_\eps-h + \frac{1}{2\pi} 
\sum_{k\ne 0} e^{ikx} \frac{\sinh(|k|(z+h_\eps)}{\sinh(|k|h_\eps)} \widehat{\eta^*\circ\zeta}(k) .  
\eq
\item[(iv)] The ``conformal depth" $h_\eps$ is given by 
\bq   \label{depth} 
h_\eps = h + \frac1{2\pi} \int_0^{2\pi} \eta^*( \zeta(x)) \ dx.  
\eq
 \end{itemize}
Here we have defined  the Fourier transform as 
\[
\widehat{f}(k)=\int_\T f(x)e^{-ikx}dx;
\]
%\item[(iv)] $\| \na_{x, z}(X-x)\|_{ L^\infty(\Rr^2_-)}+ \|\na_{x, z}(Z-z)\|_{ L^\infty(\Rr^2_-)}\le C\eps$.
 \end{theo}
 \begin{proof}   %Proposition 3.3, page 1044 in \cite{NguyenStrauss} IS INFINITE DEPTH. 
%[PEGO-SUN] IS NEITHER PERIODIC, NOR IS THEIR RIEMANN MAP ANALYTIC.   
%HERE IS A PROPER PROOF;  

{\bf 1.} Assume first that there exist $h_\eps>0$ and a conformal mapping $(X, Z): \Omega\to \Omega_\eps$ such that (i) and (ii) hold. We will prove (iii) and (iv).  For brevity, denote $f(x) = X(x,0)-x$.  We have $\Delta X=0$ and  $\partial_z X(x,-h_\eps)=0$ 
as well as $\Delta Z=0$ and  $Z(x,-h_\eps)=-h$.  By solving the Laplace problem for $X$ and using the Cauchy-Riemann equations,  we find
\bq  \label{capX}
X(x,z)-x = \frac1{2\pi} \sum_k \hat f(k) \frac{\cosh(|k|(z+h_\eps)}{\cosh(|k|h_\eps)} e^{ikx}   
\eq
and 
\bq  \label{capZ} 
Z(x,z) = z+h_\eps-h + \frac{i}{2\pi} 
\sum_{k\ne 0} \hat f(k)\ \sign(k) \frac{\sinh(|k|(z+h_\eps)}{\cosh(|k|h_\eps)} e^{ikx} .  
\eq 
Furthermore, on top we have 
\bq  \label{capZ0}
\eta^*(\zeta(x)) = Z(x,0) =    h_\eps-h + \frac{i}{2\pi} 
\sum_{k\ne 0} \hat f(k)\ \sign(k) \tanh(|k|h_\eps) e^{ikx},
\eq 
where $\zeta(x)=X(x, 0)$. Thus $h_\eps-h$ equals the average of $\eta^*\circ\zeta$, which is the same as \eqref{depth}.  
For $k\ne0$, \eqref{capZ0} gives us 
\bq\label{hatf:X}
\hat f(k) = -i\frac{\sign(k)}{\tanh(|k|h_\eps)}\ \widehat{\eta^*\circ\zeta}(k). 
\eq
Substituting  this formula for $\hat f$ into \eqref{capX} and \eqref{capZ}, and noting that $X(x, z)-x$ is odd in $x$,  we obtain \eqref{z1} and \eqref{z2}. 

Thus we have derived the formulas \eqref{z1}, \eqref{z2}, and \eqref{depth}. Moreover, by evaluating  \eqref{z1} at $z=0$ and  invoking \eqref{depth}, we obtain
\bq\label{zeta:fixedpoint}
g(x):= 
\zeta(x)-x=\frac{i}{2\pi}\sum_{k\ne 0}e^{ikx} \frac{\mathrm{sign}(k)}{\tanh(|k| h_\eps)}  
\widehat{\eta^*\circ (I+g} (k), 
\eq 
where $g$ is periodic and $Ix=x$.

{\bf 2.} Now we prove the existence of $h_\eps$ and $(X, Z)$. 
Starting with arbitrary $g\in H^s(\T)$ and $s\ge1$, we define $w=\eta^*\circ(I+g):=\sG_0(g)$, 
which occurs at the right end of \eqref{zeta:fixedpoint}.  
We notice in view of \eqref{zeta:fixedpoint} that $g=\zeta -\I$ can be defined as the fixed point of the nonlinear mapping $g \to \sF_0\sG_0$, where 
\[
\sF_0(w):= -\frac{i}{2\pi}\sum_{k\ne 0}e^{ikx} \frac{\mathrm{sign}(k)}{\tanh(|k|(h+(2\pi)^{-1}\hat{w}(0))) }\hat{w}(k). 
\]
There exists $r=r(h, s)$ sufficiently small so that  if $\| w\|_{H^s(\T)}\le r_0$ then 
$(2\pi)^{-1} |\hat{w}(0)|\le \frac{h}{2}$ and we have
\[
\| \sF_0(w)\|_{H^s(\T)}\le C_1\| w\|_{H^s(\T)},\quad C_1=C_1(s, h).
\]
On the other hand,  there exists $C_2>0$ independent of $\eps\in (0, 1)$ such that for any $g\in H^s(\T)$, we have
\[
\| \eta^*(\I+g)\|_{H^s(\T)}\le C_2\eps (1+\| g\|_{H^s(\T)}).
\]
Thus $\sF_0\sG_0$ maps the ball $B(0, r_0)\subset H^s(\T)$ to itself provided $\eps<\frac{r}{C_1C_2(1+r)}$. 

For the contraction of $\sF_0\sG_0$ we consider $w_1$, $w_2\in B(0, r)$. Using the mean value theorem for $\tanh(|k|(h+\cdot)$, we find 
\begin{align*}
&\left|\frac{\hat{w_1}(k)}{\tanh(|k|(h+(2\pi)^{-1} \hat{w_1}(0))) }-\frac{\hat{w_2}(k)}{\tanh(|k|(h+(2\pi)^{-1}\hat{w_2}(0))) }\right|\le C|\hat{w_1}(k)-\wh{w_2}(k)|\\
&\qquad+C|\hat{w_1}(k)\frac{|k|}{\cosh^2(|k|(h+\tt))}|\hat{w_1}(0)-\hat{w_2}(0)|
\end{align*}
for some $\tt$ between $(2\pi)^{-1}\hat{w_1}(0)$ and $(2\pi)^{-1}\hat{w_2}(0)$. In particular, $|\tt|\le \frac{h}{2}$, and hence 
\[
\frac{|k|}{\cosh^2(|k|(h+\tt))}\le C
\]
 since $\cosh(x)$ decays exponentially as $x\to \infty$. We deduce that 
\[
\| \sF_0(w_1)-\sF_0(w_2)\|_{H^s(\T)}\le C_3 \| w_1-w_2\|_{H^s(\T)}(1+\| w_1\|_{H^s(T)}),\quad C_3=C_3(s, h).
\]
This implies 
\begin{align*}
\| \sF_0\sG_0 (g_1))-\sF_0\sG_0(g_2))\|_{H^s(\T)}&\le C_3 \| \eta^*(\I+g_1)-\eta^*(\I+g_2)\|_{H^s(\T)}(1+\|\eta^*(\I+g_1)\|_{H^s(\T)})\\
&\le C_4\eps \| g_1-g_2\|_{H^s(\T)}(1+r).
\end{align*}
Therefore, $\sF_0\sG_0$ is a contraction on $B(0, r)\subset H^s(\T)$ provided
\[
\eps <\min\left\{\frac{r}{C_1C_2(1+r)}, \frac{1}{2C_4(1+r)}\right\}.
\]
We conclude that $g=\zeta-\I$ is the unique fixed point of $\sF_0\sG_0$ in $B(0, r)$. Using $\zeta$ we  define $h_\eps$, $X$, and $Z$ by the formulas \eqref{depth}, \eqref{z1}, and \eqref{z2}. Then it can be directly checked that $X$ and $Z$ satisfy the desired properties. 
\end{proof}

 \begin{rema}\label{rema:XZzeta}
$X$, $Z$, and $\zeta$ are  analytic in the parameter $\eps$  with values in  Sobolev spaces. Indeed, 
we see that the right side of \eqref{zeta:fixedpoint}
  is  analytic in $\eps$ with values in Sobolev spaces  since $\eta^*$ is so. The analyticity of $\zeta$ in $\eps$ then follows from the Analytic Implicit Function Theorem. Next, we return to the formula \eqref{z1} in which  $\eta^*\circ \zeta$ is now analytic in $\eps$ with values in any Sobolev space provided $\eps$ is small enough.  Consequently $\widehat{\eta^*\circ \zeta}(k)$ is analytic in $\eps$ for all $k$, and the series in \eqref{z1} converges absolutely and is analytic with values in Sobolev spaces. 
\end{rema} 
  
%After the Riemann transformation, the Dirichlet-Neumann operator  takes the form given 
%in \eqref{flatten:G} below.

% The proof of Theorem \ref{theo:flattenG} will make use of the following trace result.  
%\begin{prop}\label{prop:trace}
%Let $U\subset \Rr^n$ be Lipschitz domain with compact boundary and denote 
%\bq
%H_{\cnx}(U)=\{u\in L^2(U)^n: \cnx u\in L^2(U)\}.
%\eq
% If $u\in H_{\cnx}(U)$ and $w\in H^1(U)$, then 
%\bq\label{Stokes}
%\int_U u\cdot \na w+\int_U w\cnx u=\langle \gamma_\nu(u), \gamma_0(w)\rangle_{H^{-\mez}(\p U), H^\mez(\p U)}
%\eq
%where $\gamma_0(w)$ is the trace of $w$ and $\gamma_\nu (u)$ is the trace of $u\cdot \nu$, $\nu$ being the unit outward normal to $\p U$. The trace operator 
%\bq\label{normaltrace}
%\gamma_\nu:\quad H_{\cnx}(U)\to H^{-\mez}(\p U)
%\eq
%is continuous.
%\end{prop}
%\begin{proof}
%See Section  3.2, Chapter IV in  \cite{BoyerFabrie}.  
%\end{proof}
\subsection{The transformed Dirichlet-Neumann operator}
%%%%%%%%%%%%%%%%%%%%%%%%%%%%%%%%%
This section is devoted to the proof of the following theorem.
\begin{theo}\label{theo:flattenG} 
For fixed  $\alpha\ne 0$, denote $\beta=\alpha^2$. Fix any depth $h\in (0, \infty)$. There exists $\iota>0$ such that for each $\eps \in (-\iota, \iota)$ there exists  a bounded linear operator $\cG_{\eps, \beta}\in \cL(H^\mez(\T), H^{-\mez}(\T))$   with the following properties.  Let $s\ge\frac12$.  
\begin{itemize}
\item[(i)] 
There exists $0<\eps_0(s)\le \iota$ such that for all $\eps \in (-\eps_0(s), \eps_0(s))$ 
the operator  $\cG_{\eps, \beta}$ is bounded from   $\cL(H^s(\T)$ to $H^{s-1}(\T))$. 
In addition, we have
\bq\label{cG:even}
(\cG_{\eps,\beta}(\bar f(-\cdot)))(x) = \overline{(\cG_{\eps,\beta} f)(-x)}
\eq
and 
\bq\label{form:G0}
\cG_{0, \beta}=\sqrt{|D|^2+\beta}\ \tanh(h\sqrt{|D|^2+\beta}).
\eq
If $s=1$, the operator $\cG_{\eps, \beta}$ is self-adjoint on $L^2(\T)$.  
\item[(ii)] 
The mapping
\bq\label{mappimg:cG}
 (-\eps_0(s), \eps_0(s))\times (0, \infty)\ni  (\eps,\beta) \mapsto \cG_{\eps, \beta}\in \mathcal{L}(H^{s}(\T), H^{s-1}(\T))
\eq
 is analytic.

  \item[(iii)]  If $s\ge 1$ and $f_0\in H^s(\T)$, we have the identity 
\bq\label{flatten:G}
\cG_{\eps, \beta}(f_0\circ \zeta)(x)=e^{-i\alpha y}\zeta'(x)\big[ G(\eta^*)(f_0e^{i\alpha y})\big]\vert_{(\zeta(x), y)}
\eq
for a.e. $x\in \T$ and for all $y\in \Rr$. 
\end{itemize}
We emphasize that the operator $\cG_{\eps, \beta}$ and the constants $\iota$ and $\eps_0(s)$ all depend on the depth $h$.
\end{theo}
%The coefficients $C_k^\pm$ above are given in \eqref{coeff:C} and the coefficients $B$ 
%in \eqref{coeff:B}.  
%{\red WE NEED THE COEFFICIENTS $E_k$ ETC. FROM MATHEMATICA!}   
%We emphasize that the expression \eqref{flatten:G} is independent of $y$.  
%Since $\alpha$ only appears squared in \eqref{flatten:G}, we denote $\beta=\alpha^2$.  
%By virtue of Proposition \ref{prop:flattenG}, we have the expansion
%\bq
%e^{-i\alpha y}\zeta'(x)\big[ G(\eta^*)(\wt{v_2}(x)e^{i\alpha y})\big]\vert_{\zeta(x)}=\Om(D)w_2(x)+\alpha^2 \eps R_1w_2(x)  +\alpha^2 \eps^2 R_2w_2(x) 
%+ \alpha^2O(\eps^3).  
%\eq 
%%%%%%%%%%
%\begin{rema}
In the infinite depth case, Theorem \ref{theo:flattenG} was proven in Theorem 2.7, \cite{CreNguStr}. The assertions (i) and (iii) can be obtained similarly to the infinite depth case. Regarding the analyticity of $G_{\eps, \beta}$ in $(\eps, \beta)$, the direct proof in \cite{CreNguStr} can be adapted to the finite depth case. However,  the low regularity  case $s=\mez$ would require a careful treatment, due to the $\eps$-dependence of the  conformal domain $\{-h_\eps<z<0\}$ and the  trace operator for vector fields not being in $H_{\cnx}$ (see \eqref{def:Hdiv}) which arise if one changes  $z\mapsto z\frac{h}{h_\eps}$ to cope with $h_\eps$. We will provide an alternative, shorter  proof using the  analytic implicit  function theorem, in the spirit of the proof of Lemma 2.2 in \cite{Groves}. We will need the following classical trace operator. 
%\end{rema}
\begin{prop}[\protect{Section  3.2, Chapter IV, \cite{BoyerFabrie}}]\label{prop:trace}
Let $U\subset \Rr^n$ be Lipschitz domain with compact boundary and denote 
\bq\label{def:Hdiv}
H_{\cnx}(U)=\{u\in L^2(U; \Rr^n): \cnx u\in L^2(U)\}.
\eq
There exists a bounded linear operator $\gamma: H_{\cnx}(U)\to H^{-\mez}(\p U)$ such that $\gamma(u)=u\cdot \nu$ for $u\in C^\infty_0(\overline{U})$, where $\nu$ is the unit outward normal to $\p U$. In particular, there exists a constant $C>0$ such that 
\bq\label{cont:tracegamma}
\| \gamma(u)\|_{H^{-\mez}(\p U)}\le C\Big(\| u\|_{L^2(U)}+\| \cnx u\|_{L^2(U)}\Big)\quad\forall u\in H_{\cnx}(U).
\eq
 Moreover, if $u\in H_{\cnx}(U)$ and $w\in H^1(U)$, then the Stokes formula 
\bq\label{Stokes}
\langle \gamma(u), \gamma_0(w)\rangle_{H^{-\mez}(\p U), H^\mez(\p U)}=\int_U u\cdot \na w+\int_U w\cnx u
\eq
holds, where $\gamma_0(w)$ is the trace of $w$.% The trace operator 
\end{prop}
{\bf Proof of Theorem \ref{theo:flattenG}}
\\
 We assume without loss of generality that  $\alpha>0$ and denote $\T^2_\alpha=\T\times (\Rr/ 2\pi \alpha \Zz)$. %By virtue of  Theorems 3.8 and 3.12 in \cite{ABZ3},  if we view $\eta_*$ as a function on $\T^2_\alpha$ that is independent of $y$, then
%\bq\label{cont:DN}
%G(\eta^*)\in \cL(H^s(\T^2_\alpha), H^{s-1}(\T^2_\alpha))\quad \forall\ s \ge \tfrac12. 
%\eq
Let $(X, Z): \Omega_\eps \to \Omega$ be the Riemann mapping  in Theorem \ref{prop:Riemann}. It has the Jacobian $\cJ=\cJ(x, z, \eps)=|\p_xX|^2+|\p_z X|^2$, which is analytic in $\eps$ near $0$. For an arbitrary function $f\in H^\mez(\T)$,  we will define 
the function $\Tt_f$ by solving 
\bq\label{system:Theta}  
  \begin{cases}
\Delta_{x, z} \Tt_f - \beta\cJ\Tt_f=0\quad\text{ in  } \Omega_\eps=\{(x, z)\in \Rr^2: x\in \T,~-h_\eps<z<0\},\\
\Tt_f(x, 0)=f(x),\\
\p_z\Theta_f(x,-h_\eps)=0. 
\end{cases}
\eq
Then we define the operator $\cG_{\eps,\beta}$ by
\bq\label{def:cG}
\cG_{\eps, \beta} f=\p_z\Tt_f(\cdot, 0),%= \p_z\mathfrak Tf \Big|_{\{z=0\}}
\eq
where $\p_z\Tt_f(\cdot, 0)$ is the trace  $\gamma(\na \Tt_f)$ given in Proposition \ref{prop:trace}.
As shown in  the proof of Theorem 2.7 in \cite{CreNguStr}, for any $s\ge \mez$ there exists $\eps_0(s)>0$ such that  
\bq\label{bounded:cT}
\cG_{\eps, \beta}\in \cL(H^s(\T), H^{s-1}(\T))\quad \forall \eps \in (-\eps_0(s), \eps_0(s));
\eq 
moreover, $ \cG_{\eps, \beta}$ is self-adjoint on $L^2$ and it satisfies \eqref{cG:even} thanks to the fact that $\eta^*$ is even. 

 In order to prove \eqref{flatten:G}, we consider the problem 
\bq\label{sys:vartheta}
\begin{cases}
\Delta_{X, Y, Z}\vartheta=0\quad\text{in } \{(X, Y, Z)\in \Rr^3: (X, Y)\in \T^2_\alpha,~-h<Z<\eta^*(X)\},\\
\vartheta(X, Y, \eta^*(X))=f_0(X)e^{i\alpha Y},\\
\p_Z\vartheta(X, Y,  -h)=0.
\end{cases}
\eq 
%that includes some periodic $Y$-dependence.  
For clarity, we are now denoting the physical variables by $(X,Y,Z)$.  
By \eqref{def:Gh}, we have
\bq\label{G:f0e}
G(\eta^*)(f_0e^{i\alpha Y})(X, Y)=\p_Z\vartheta(X, Y, \eta^*(X))-\p_X\vartheta(X, Y, \eta^*(X))\p_X\eta^*(X).
\eq
The solution of \eqref{sys:vartheta} is given by  $\vartheta(X, Y, Z)=\tt(X, Z)e^{i\alpha Y}$, where $\tt(X, Z)$ solves 
\bq\label{system:tt}
\begin{cases}
\Delta_{X,  Z}\tt -\alpha^2\tt=0\quad\text{in } \{(X, Z)\in \Rr^2: X\in \T,~-h<Z<\eta^*(X)\},\\
\tt(X, \eta^*(X))=f_0(X),\\
\p_Z\tt(X, -h)=0.
\end{cases}
\eq
Using the Riemann mapping $(X, Z): \Omega_\eps \to \Omega$ in Proposition \ref{prop:Riemann} to flatten the 
domain $\Omega$, we find that $\Tt(x, z):=\tt(X(x, z), Z(x, z))$ solves \eqref{system:Theta}  with $f=f_0\circ \zeta$, i.e. $\Tt=\Tt_{f_0\circ \zeta}$. Then, for $s>1$,  the identity \eqref{flatten:G} follows by using \eqref{G:f0e}, the chain rule, and the Cauchy-Riemann equations. The case $s=1$ is treated by an approximation argument. We refer to the proof of Theorem 2.7 in \cite{CreNguStr} for further details. 

Assume now that $\eps=0$, i.e. $\eta^*=0$. Then we have $\cJ=1$ and $h_\eps=h$ in view of \eqref{depth}. Consequently, by solving \eqref{system:Theta} using  Fourier transform in $x$, we find
\bq\label{pzTt:flat}
\p_z\Tt_f(\cdot, 0)=\sqrt{|D|^2+\beta}\ \tanh(h\sqrt{|D|^2+\beta})f.
\eq
This proves \eqref{form:G0}.

The remainder of this proof is devoted to the analyticity \eqref{mappimg:cG} of $\cG_{\eps, \beta}$. We denote
\[
\Omega_0:=\{(x, z): x\in \T,~z\in (-h, 0)\}.
\]
We fix $\beta_0>0$ and $f\in H^s(\T)$ with $s\ge \mez$. We first prove that for any  $F_0\in H^{s-\mez}(\Omega_0)$ and $G\in (H^{s-\mez}(\Omega_0))^2$, the problem 
\bq\label{system:defS}  
  \begin{cases}
\Delta_{x, z} v -\beta_0 v=F_0+\cnx G\quad\text{ in  } \Omega_0,\\
v(x, 0)=f(x),\\
\p_zv(x,-h)=G_2(x, -h)
\end{cases}
\eq
has a unique solution $v\in H^{s+\mez}(\Omega_0)$ in the following sense: $v\vert_{\{z=0\}}=f$ in the trace sense and  
\bq\label{weakform:v}
\int_{\Omega_0} \na_{x, z}v\cdot \na_{x, z}\varphi + \beta_0 v\varphi \,dxdz=\int_{\Omega_0} -F_0\varphi + G\cdot \na_{x, z}\varphi \, dxdz\quad\forall \varphi \in H^{1, 0}(\Omega_0),
\eq
where $H^{1, 0}(\Omega_0)=\{u\in H^1(\Omega_0):~u\vert_{\{z=0\}}=0\}$. Note that \eqref{weakform:v} implies $\Delta_{x, z}v-\cnx G=\cnx( \na_{x, z}v-G)=\beta_0v+F_0\in L^2(\Omega_0)$ in the sense of distributions, and hence 
\[
\int_{\Omega_0} (\na_{x, z}v-G)\cdot \na_{x, z}\varphi +\cnx(\na_{x, z}v-G)\varphi\, dxdz=0\quad\forall \varphi \in H^{1, 0}(\Omega_0).
\]
In view of the Stokes formula \eqref{Stokes}, we deduce that the normal trace of $\na_{x, z}v-G$ at $\{z=-h\}$ vanishes in $H^{-\mez}(\T)$, justifying the Neumann condition $\p_zv(x,-h)=G_2(x, -h)$.  

 For $s=\mez$, the existence and uniqueness of $v\in H^1(\Omega_0)$ follow from the Lax-Milgram theorem and the lifting operator from $H^s(\T^d)$ to $H^{s+\mez}(\Omega_0)$. Moreover, we have 
\[
\| v\|_{H^1(\Omega_0)}\le  C(\| F_0\|_{L^2(\Omega_0)}+\| G\|_{(L^2(\Omega_0))^2}+\| f\|_{H^\mez(\T)}).
\]
For $s>\mez$, it remains to obtain the regularity  $v\in H^{s+\mez}(\Omega_0)$. We decompose  $v=v_1+v_2$, where $v_1\in H^1(\Omega_0)$ solves 
\bq\label{system:defS:v1}  
  \begin{cases}
\Delta_{x, z} v_1 -\beta_0 v_1=F_0\quad\text{ in  } \Omega_0,\\
v_1(x, 0)=f(x),\\
\p_zv_1(x,-h)=0,
\end{cases}
\eq
and $v_2\in H^1(\Omega_0)$ solves 
\bq\label{system:defS:v2}  
  \begin{cases}
\Delta_{x, z} v_2-\beta_0 v_2=\cnx G\quad\text{ in  } \Omega_0,\\
v_2(x, 0)=0,\\
\p_zv_2(x,-h)=G_2(x, -h).
\end{cases}
\eq
Since $F_0\in H^{s+\mez}(\Omega_0)$ has nonnegative regularity, combining standard elliptic regularity with linear interpolation yields $v_1 \in H^{s+\mez}(\Omega_0)$. This is also the case for $v_2$ provided $s\ge \tdm$, so that $\cnx G\in H^{s-\tdm}(\Omega_0)$ has nonnegative regularity. Since the mapping $H^{s-\mez}(\Omega_0)\ni G\mapsto v_2\in H^{s+\mez}(\Omega_0)$ is linear, the remaining  case $s\in (\mez, \tdm)$  follows by linear interpolation between $s=\mez$ and $s=\tdm$.

 We denote the mapping $(F_0, G, f)\mapsto v$ by 
\bq\label{def:cS}
v=\mathcal{S}(F_0, G, f).
\eq
Since 
\[
\mathcal{S}(F_0, G, f)=\mathcal{S}(0, 0, f)+\mathcal{S}(F_0, 0, 0)+\mathcal{S}(0, G, 0)
\]
and each operator on the right-hand side is linear in its respective argument, it follows that 
\bq\label{analytic:cS}
\mathcal{S}: H^{s-\mez}(\Omega_0)\times (H^{s-\mez}(\Omega_0))^2\times H^s(\T^d)\to H^{s+\mez}(\Omega_0)\quad\text{is~ analytic}. 
\eq

 Returning to the problem \eqref{system:Theta}, in order to cope with the $\eps$-dependence of $\Omega_\eps$, we make the simple rescaling 
\[
\Theta(x, z; \eps, \beta)=\Psi\left(x, \frac{h}{h_\eps}z; \eps, \beta\right),
\quad \cJ(x, z ;\eps)=\tilde{J}\left(x, \frac{h}{h_\eps}z; \eps\right),
\]
so that $\Psi\equiv \Psi(\eps, \beta): \Omega_0\to \Rr$ satisfies 
\bq\label{system:Psi0}  
  \begin{cases}
\left(\p_x^2+\big(\frac{h}{h_\eps}\big)^2\p_z^2\right) \Psi- \beta\tilde{J}\Psi=0\quad\text{ in  } \Omega_0\ , \\ 
\Psi(x, 0)=f(x),\quad \p_z\Psi(x,-h)=0.
\end{cases}
\eq
Use of the chain rule in the definition \eqref{def:cG} of $\cG_{\eps, \beta}$ yields 
  \bq\label{cG:10}
  \tfrac{h_\eps}{h} \cG_{\eps,\beta}f = \p_z\Psi(\eps, \beta)\vert_{z=0}.
\eq
We rewrite \eqref{system:Psi0} as
\bq\label{system:Psi1}  
  \begin{cases}
\Delta_{x, z}\Psi-\beta_0 \Psi=(\beta\tilde{J}-\beta_0)\Psi+ \Big[1-\big(\frac{h}{h_\eps}\big)^2\Big]\p_z^2 \Psi \quad\text{ in  } \Omega_0\ , \\ 
\Psi(x, 0)=f(x),\quad \p_z\Psi(x,-h)=0. 
\end{cases}
\eq
Denoting $F_0=F_0(\Psi, \eps, \beta)=(\beta\tilde{J}-\beta_0)\Psi$ and 
\[
G=G(\Psi, \eps)=\begin{bmatrix} 0 \\  \big[1-\big(\frac{h}{h_\eps}\big)^2\big]\p_z \Psi\end{bmatrix},
\]
we  rewrite \eqref{system:Psi1} in the form of \eqref{system:defS}:
   \bq\label{system:Psi3n}  
  \begin{cases}
\Delta_{x, z}\Psi-\beta_0 \Psi= F_0+ \cnx G \quad\text{ in  } \Omega_0\ ,  \\ 
\Psi(x, 0)=0,\quad \p_z\Psi(x,-h)=G_2,
\end{cases}
\eq 
where the Neumann condition $\p_z\Psi(x,-h)=G_2= \big[1-\big(\frac{h}{h_\eps}\big)^2\big]\p_z \Psi$ is equivalent to $\p_z\Psi(x,-h)=0$.  Hence, in terms of the analytic operator $\mathcal{S}$ \eqref{def:cS}, we have  
\[
\Psi=\cS(F_0(\Psi, \eps, \beta), G(\Psi, \eps), f).
\]
For fixed $f\in H^s(\T)$, we denote 
\[
\cT(\Psi, \eps, \beta)=\Psi- \cS(F_0(\Psi, \eps, \beta), G(\Psi, \eps), f).
\]
Both $F_0$ and $G$ are linear in $\Psi$ and analytic in $(\eps, \beta)\in \mathcal{O}_s$, a neighborhood of $(0, \beta_0)$. Together with the analyticity \eqref{analytic:cS} of $\mathcal{S}$, we deduce that  $\cT:  H^{s+\mez}(\Omega_0)\times \mathcal{O}_s\to H^{s+\mez}(\Omega_0)$ is analytic. Set $\Psi_0=\mathcal{S}(0, 0, f)$. Since $F_0(\cdot, 0, \beta_0)\equiv 0$ and $G(\cdot,  0)\equiv 0$, we have $\cT(\Psi_0, 0, \beta_0)=0$ and $D_1\cT(\Psi_0, 0, \beta_0)=\mathrm{I}$. Therefore, by virtue of the analytic implicit function theorem, we deduce that 
\bq\label{analyticity:Psi}
\Psi=\Psi(\eps, \beta): \mathcal{O}_s\ni(\eps, \beta)\to H^{s+\mez}(\Omega_0)\quad\text{is~analytic},
\eq
where $\mathcal{O}_s$ can be shrunk if necessary. 

Next, we take the trace $\p_z\Psi(\eps, \beta)\vert_{z=0}$ to make the connection to $\cG_{\eps, \beta}f=\frac{h}{h_\eps}\p_z\Psi(\eps, \delta)\vert_{z=0}$.  If $s>1$, then combining \eqref{analyticity:Psi} (and a similar property for $\Psi_2$) with  the continuity of the standard trace operator $H^{s-\mez}(\Omega_0)\to H^{s-1}(\p \Omega_0)$, we deduce that 
\bq\label{analyticity:tracePsi}
\p_z\Psi(\eps, \beta)\vert_{z=0}: (\eps, \beta)\ni \mathcal{O}_s\to H^{s-1}(\T)\quad\text{is~analytic}
\eq
for some neighborhood $\mathcal{O}_s$ of $(\eps, \beta)=(0, \beta_0)$.  To treat the case  $s=\mez$ we consider the vector field 
\[
u=\begin{bmatrix}
\p_x\Psi(\eps, \beta)\\ \big(\frac{h}{h_\eps}\big)^2\p_z\Psi(\eps, \beta)
\end{bmatrix}.
\]
 Because $\cnx u=\beta \tilde{J} \Psi(\eps, \beta)\in H^1(\Omega_0)\subset L^2(\Omega_0)$ in view of \eqref{system:Psi0}, we have  that $u\in H_{\cnx}(\Omega_0)$, defined by \eqref{def:Hdiv}. Therefore, by virtue of Proposition \ref{prop:trace}, the trace $\gamma(u)=:\big(\frac{h}{h_\eps}\big)^2\p_z\Psi(\eps, \beta)\vert_{z=0}$
 is well-defined in $H^{-\mez}(\T)$. Therefore, \eqref{analyticity:tracePsi} is still valid for $s=\mez$. Viewing $\Psi(\eps, \beta)$ as the bounded linear operator  $f\mapsto \Psi$ solution of \eqref{system:Psi0}, we deduce with the aid of linear interpolation  that the map
 \[
 f\mapsto \p_z\Psi(\eps, \beta)\vert_{z=0}: (\eps, \beta)\ni \mathcal{O}_s\to \cL(H^s(\T), H^{s-\mez}(\T))
 \]
is analytic for all $s\ge \mez$. This concludes the proof of the analyticity \eqref{mappimg:cG} of $\cG_{\eps, \beta}$. $\square$

\subsection{The reformulated instability problem}
Theorem \ref{theo:flattenG} allows us to precisely formulate the instability problem as follows.  
Starting from the Riemann stretch \eqref{def:zeta}, 
it is convenient to first define the two auxiliary operators 
\bq
\zeta_{\sharp} f(x)=(f\circ \zeta_{})(x),\quad \zeta_{*}f(x)=\zeta_{}'(x)\zeta_{\sharp} f(x).
\eq
Then we define the new unknowns 
\bq\label{def:w12}
w_1(x)=\zeta_{*}\wt{v_1},\quad w_2(x)=\zeta_{\sharp} \wt{v_2}.
\eq
We apply $\zeta_{*}$ to \eqref{eq:v1} and $\zeta_{\sharp}$ to \eqref{eq:v2}, and use Theorem \ref{theo:flattenG} (iii) to have 
\[
\zeta_{}'(x)\big[ G(\eta^*)(\wt{v_2}e^{i\alpha y})\big]\vert_{(\zeta_{}(x), y)}=e^{i\alpha y}\cG_{\eps,\beta}(\wt{v_2}\circ \zeta_{})(x),
\]
that is, $\zeta_{*}G(\eta^*)(v_2)(x, y)=e^{i\alpha y}\cG_{\eps,\beta}(w_2)(x)$. We  thus obtain the equivalent linearized system
\begin{align}\label{eq:w1}
&\p_t w_1=\p_x\big(p_{}(x)w_1\big) 
+ \cG_{\eps,\beta} w_2,\\  \label{eq:w2}
&\p_t w_2=-\frac{1 +q_{}(x)}{\zeta_{}'(x)}w_1+p_{}(x)\p_xw_2,
\end{align}
 where  the variable coefficients 
\bq\label{def:pq}
p_{}:=\frac{c^*-\zeta_{\sharp} V^*}{\zeta_{}'},\quad q_{}:=-p_{}\p_x(\zeta_{\sharp} B^*)
\eq
have auxiliary dependence on $\eps$ and $h$.  

For the spectral analysis, we seek solutions of the linearized system \eqref{eq:w1}--\eqref{eq:w2} 
of the form $w_j(x, t)=e^{\ld t}u_j(x)$, thereby arriving at the  {\it eigenvalue problem}  
\bq\label{eiproblem}
\ld U=\cL_{\eps,\beta} U:= 
\begin{bmatrix} 
\p_x(p_{}(x)\cdot)  & \cG_{\eps,\beta}\\
-\frac{1 +q_{}(x)}{\zeta_{}'(x)}& p_{}(x)\p_x
\end{bmatrix}U,\quad U=
\begin{bmatrix}  u_1 \\ u_2 \end{bmatrix}.
\eq
Of course $\cL_{\eps, \beta}$ depends also on the depth $h$ but we will drop this dependence to ease  notation. By virtue of Theorem \ref{theo:flattenG} (i), we have $\cG_{\eps,\beta}\in \cL(H^1(\T), L^2(\T))$ is self-adjoint for $\eps\in (-\eps_0(1), \eps_0(1))$, $\beta\in (0, \infty)$, and $h > 0$. Therefore, we regard $\cL_{\eps,\beta}$, defined above,  as a bounded linear operator from  $(H^1(\T))^2$ to $(L^2(\T))^2$.   Moreover, $\cL_{\eps,\beta}$ can be written in Hamiltonian form
 \bq\label{def:Hamiltonian}
\cL_{\eps,\beta}=J\cH_{\eps,\beta},\quad J=\begin{bmatrix} 0& 1\\ -1 & 0\end{bmatrix},\quad \cH_{\eps,\beta}=\begin{bmatrix} \frac{1 +q_{}(x)}{\zeta_{}'(x)} & -p_{}(x)\p_x \\ \p_x(p_{}(x)\cdot) & \cG_{\eps,\beta}\end{bmatrix},
 \eq
 where $\cH_{\eps,\beta}$ is self-adjoint. All entries of $\cH_{\eps,\beta}$ depend on $\eps$ and $h$ but only the lower right corner depends on $\beta$. For the two-dimensional modulational instability,  $\beta=0$ and it was shown in \cite{NguyenStrauss} that $\cG_{\eps, 0}=|D|\tanh(h|D|)$ (see also Proposition \ref{prop:Rj} below).  However, when $\beta \ne 0$, $\cG_{\eps,\beta}$ is no longer a Fourier multiplier but a genuine pseudo-differential operator, which will be expanded up to $O(\eps^3)$ in Section \ref{Sec:R}. 
  \begin{defi}
 The Stokes wave $(\eta^*, \psi^*, c^*)$ is said to be unstable with respect to  transverse perturbations if there exists $\beta > 0$ (and, hence, a transverse wave number $\alpha = \sqrt{\beta}$) such that $\cL_{\eps,\beta}$  has an eigenvalue with positive real part. 
 \end{defi}

 The expansions in powers of $\eps$ for the conformal depth $h_{\eps}$, Riemann stretch $\zeta_{}$, and the coefficients in \eqref{eq:w1}--\eqref{eq:w2} are given in the next proposition.  
 %%%%%%%%%%%%%%%%%
\begin{prop}\label{prop:expandpq}
The conformal depth $h_{\eps}$ is real analytic in $\eps$ with expansion
\begin{align}   \label{heps expansion}
h_{\eps} = h + h_2\eps^2 + O\left(\eps^4\right),
\end{align}
where 
\begin{align}
h_2 = \frac{c_0^4-3}{4c_0^2} \quad \textrm{and} \quad c_0 = \tanh^\mez(h).
\end{align}
Moreover, for any $s>\frac52$, the following functions are analytic in $\eps$ with values in $H^s(\T)$ provided $\eps$ is sufficiently small. They have the expansions:
\begin{align}\label{zeta1}
&\zeta_{}(x)=x+\eps\zeta_{1,1} \sin x+\eps^2 \zeta_{2,2} \sin(2x)+ \eps^3\Big[\zeta_{3,1}\sin x + \zeta_{3,3}\sin(3x) \Big]+ O(\eps^4),\\%\quad \zeta^1(x)= \sin x,~\zeta^2=\sin(2x),\\ 
 \label{expand:p}
&p_{}(x)=c_0 +\eps p_{1,1}\cos x+\eps^2\Big[p_{2,0}+p_{2,2}\cos(2x)\Big]+\eps^3\Big[p_{3,1}\cos x+p_{3,3}\cos(3x)\Big]+O(\eps^4),\\ \label{expand:q}
&q_{}(x)=\eps q_{1,1}\cos x+\eps^2\Big[q_{2,0}+q_{2,2}\cos(2x)\Big]+\eps^3\Big[q_{3,1}\cos x+q_{3,3}\cos(3x)\Big]+O(\eps^4),\\ \label{expand:qzeta}
&\frac{1+q_{}(x)}{\zeta_{}'(x)}=1+\eps r_{1,1}\cos x+\eps^2\Big[r_{2,0}+r_{2,2}\cos(2x)\Big]+\eps^3\Big[r_{3,1}\cos x+r_{3,3}\cos(3x)\Big]+O(\eps^4),
\end{align}
where 
\allowdisplaybreaks
\begin{align}
\zeta_{1,1} &= c_0^{-2}, \\ \zeta_{2,2} &= \frac{c_0^8 + 4c_0^4 + 3}{8c_0^8}, \\ \zeta_{3,1} &= \frac{4c_0^{14}+2c_0^{12}-17c_0^{10}-14c_0^8+10c_0^6+10c_0^4-15c_0^2-12}{16c_0^{10}(c_0^2+1)} \\ \zeta_{3,3} &= \frac{3c_0^{12}+43c_0^8+41c_0^4+9}{64c_0^{14}} \\
p_{1,1} &= -2c_0^{-1}, \\
p_{2,0} &= \frac{-2c_0^{12}+5c_0^8+12c_0^4+9}{16c_0^7}, \\
p_{2,2} &= -\frac{c_0^4 + 3}{2c_0^7}, \\
p_{3,1} &= \frac{-2c_0^{14}+14c_0^{10}+11c_0^8-10c_0^6-10c_0^4+24c_0^2+21}{8c_0^9(c_0^2+1)}, \\
p_{3,3} &= -\frac{c_0^{12}+17c_0^8+51c_0^4+27}{32c_0^{13}} \\
q_{1,1} &= -c_0^2, \\
q_{2,0} &= 1, \\
q_{2,2} &= 2-3c_0^{-4}, \\
q_{3,1} &= \frac{4c_0^{14}+6c_0^{12}-9c_0^{10}-12c_0^8-30c_0^6-30c_0^4+69c_0^2+66}{16c_0^6(c_0^2+1)}, \\
q_{3,3} &= -\frac{3(3c_0^{12}+19c_0^8-71c_0^4+81)}{64c_0^{10}}, \\
r_{1,1} &= -(c_0^2+c_0^{-2}), \\
r_{2,0} &= \frac32 + \frac12 c_0^{-4}, \\
r_{2,2} &= \frac{9c_0^8-14c_0^4-3}{4c_0^8}, \\
r_{3,1} &= \frac{4c_0^{18}+6c_0^{16}-11c_0^{14}-12c_0^{12}-45c_0^{10}-48c_0^8+93c_0^6+90c_0^4+27c_0^2+24}{16c_0^{10}(c_0^2+1)}, \\
r_{3,3} &= \frac{-c_0^{16}-98c_0^{12}+252c_0^8-318c_0^4-27}{64c_0^{14}}.
\end{align}
%{\red Clearly each of the four remainder terms are $\red analytic$ in a neighborhood of $\eps=0$. }  
\end{prop}

\begin{proof}
The proof of Proposition  
\ref{prop:expandpq} is given in Appendix A.2 of \cite{Berti3}, noting that their $\verb|f|_{\eps}$, $\mathfrak{p}_{\eps}$, $p_{\eps}$, and $a_{\eps}$ correspond to $(h_{\eps}-h)$, $(\zeta_{}-x)$, $(p_{}-c_0)$, and $[(1+q_{})/\zeta_{}' - 1]$, respectively, in this work.
\end{proof}
Combining Theorem \ref{theo:flattenG} (with $s=1$) and Proposition \ref{prop:expandpq}, 
we conclude that the mapping
\bq\label{analyticity:cL}
(-\eps_0(1), \eps_0(1))\times (0, \infty) \ni (\eps, \beta)\mapsto \cL_{\eps,\beta} \in \cL\big((H^1(\T))^2, (L^2(T))^2\big)
\eq
is analytic for each $h > 0$. 
%%%%%%%%%%%% NEW SECTION %%%%%%%%%%%%%%%%%%%%
%%%%%%%%%%%%%%%%%%%%%%%%%%%%%%%%%%%%%%%%
\section{Resonance condition and Kato's perturbed basis} \label{Sec:Kato} 
\subsection{Resonance condition}
We begin with the basic case $\eps=0$.  Using the expansions \eqref{expand:p} and \eqref{expand:q} for $p$ and $q$ and the formula \eqref{form:G0} for $\cG_{0,\beta}$,  we find in view of \eqref{eiproblem} that
\bq
\cL_{0,\beta}=\begin{bmatrix} 
c_0\p_x  & (|D|^2+\beta)^\mez\tanh\left(h\left(|D|^2+\beta\right)^\mez\right)\\
-1& c_0\p_x
 \end{bmatrix}.
\eq
 The {\it spectrum of} $\cL_{0,\beta}$ consists of the purely imaginary eigenvalues 
\bq\label{ld0pm}
\ld^0_\pm(k, \beta)=i\left[c_0k\pm (k^2+\beta)^\frac14\tanh^\mez\left(h\left(k^2+\beta \right)^\mez \right)\right],\quad k\in \Zz,
\eq
which are the roots of the quadratic characteristic polynomial 
\bq
\Delta_0(\ld; k, \beta)=(\ld-ic_0k)^2+(k^2+\beta)^\mez\tanh\left(h\left(|D|^2+\beta\right)^\mez\right)=[\ld-\ld^0_+(k, \beta)][\ld-\ld^0_-(k, \beta)].
\eq
We learn from \cite{MMMSY} that the resonance condition is a double eigenvalue 
$\ld^0_+(-(m+1), \beta)=\ld^0_-(m, \beta)$.  That is, 
\bq
-c_0(m+1)+((m+1)^2+\beta)^\frac14\tanh^\mez\left(h\left((m+1)^2+\beta \right)^\mez \right)=c_0m-(m^2+\beta)^\frac14\tanh^\mez\left(h\left( m^2+\beta\right)^\mez \right). \label{resonancecond}
\eq
%%%%%%%%%%%%%%%%%%%%
For our purposes, we choose the {\it simplest resonance} $m=1$,  which  leads to instabilities with the largest growth rates for small Stokes waves as suggested by  numerical computations \cite{MMMSY}.  Then, we have the following proposition on the existence, uniqueness and asymptotics of the  resonant value of $\beta$.
\begin{prop} \label{prop:rescond} For each $h > 0$, there exists a unique $\beta_* = \beta_*(h)\in (0, \infty)$ such that \eqref{resonancecond} is satisfied and \\\\ (i)  $0 < \beta_*(h) < 3$, \\
 (ii) $\beta_*(h)$ is real analytic for $h > 0$, \\
(iii) $\beta_*(h) = \frac43 h^2 + \mathscr{o}(h^2)$ as $h \rightarrow 0^+$, and \\
(iv) $\beta_*(h) = \beta_{*,\infty} -\frac{12e^{-2h}}{\left(1+\beta_{*,\infty} \right)^{-\frac34} + \left(4+\beta_{*,\infty} \right)^{-\frac34}} + \mathscr{o}\left(e^{-2h}\right) $ as $h \rightarrow \infty$, where $\beta_{*,\infty} = 2.7275...$ is the root of \eqref{betaStarInf}.
\end{prop}

\begin{proof}
   Define the function 
    \begin{align}
F(\beta,h) = 3\tanh^\mez(h) - (1+\beta)^{\frac14}\tanh^\mez\left(h\left(1+\beta\right)^\mez
\right)  - (4+\beta)^{\frac14}\tanh^\mez\left(h\left(4+\beta\right)^\mez
\right)  \label{Fbetah} 
\end{align}
for $\beta, \ h>0$. For fixed $h>0$, the zeros of $F(\cdot, h)$ correspond to the solution of the resonance condition \eqref{resonancecond} with $m = 1$. 
We note that $F$ is real analytic for $\beta>0,\ h>0$.
We have 
\begin{align*}
F(0,h) =  2\tanh^\mez(h)  - 2^\mez \tanh^\mez\left(2h\right) > 0,
\end{align*}
since $\tanh(2h)/\tanh(h) < 2$ valid for $h > 0$.
On the other hand, for any $h > 0$,
\begin{align*}
F(3,h) &= 3\tanh^\frac12(h) - 4^{\frac14}\tanh^{\frac12}(2h) - 7^{\frac14} \tanh^\frac12\left(7^{\frac{1}{2}} h\right),  \\
&< 3\tanh^\frac12(h) - 4^\frac14 \tanh^\frac12(h) - 7^\frac14 \tanh^\frac12(h), \\
&< \left(3-4^\frac14 - 7^\frac14\right)\tanh^\frac12(h), \\
&< 0,
\end{align*}
where the first inequality follows from the monotonicity of $\tanh^\frac12(\xi)$ for $\xi > 0$. Consequently, for each $h > 0$, there exists $0 < \beta_* < 3$ such that $F(\beta_*,h) = 0$. 
Because $\partial_\beta F<0$,  we deduce that $\beta_*(h)$ is the unique positive zero of $F(\cdot, h)$ . 
This proves (i).  By the Analytic Implicit Function Theorem, $\beta_*(h)$ is real analytic for $h > 0$, proving (ii).

We now prove (iii).  By (i), $h (1+\beta_*(h))^\frac12  =O(h)$ and $h (4+\beta_*(h))^\frac12 =O(h)$  as $h \rightarrow 0^+$ or $h\to \infty$. Consequently, using the expansion \begin{align*}
\tanh^\frac12(\xi) = \xi^\frac12\left(1 - \frac{\xi^2}{6} + O\left(\xi^4 \right) \right) \quad \textrm{as} \quad \xi \rightarrow 0^+, \end{align*}
we obtain
\begin{align}
F(\beta_*(h),h) &= \left[3 - \left(1+\beta_*(h) \right)^\frac12 - \left(4 + \beta_*(h) \right)^\frac12 \right]h^\frac12 - \frac16 \left[3 - \left(1 + \beta_*(h) \right)^\frac32 - \left(4+\beta_*(h) \right)^\frac32 \right]h^{\frac52} \nonumber \\ &\quad + O\left(h^\frac92 \right) \quad \textrm{as} \quad h \rightarrow 0^+.\label{Fbetah0}
\end{align}
Given $F(\beta_*(h),h) =0$, it follows from \eqref{Fbetah0} that
\begin{align*}
3 - \left(1+\beta_*(h)\right)^\frac12 - \left(4+\beta_*(h) \right)^\frac12 = O(h^2) \quad \textrm{as} \quad h \rightarrow 0^+.
\end{align*}
Rearranging this result, we have
\begin{align*}
3 - (1+\beta_*(h))^\frac12 + O(h^2) &= (4+\beta_*(h))^\frac12,
\end{align*}
and upon squaring both sides, while using the bound $0 < \beta_*(h) < 3$, we find
\begin{align*}
\left(1+\beta_*(h)\right)^\frac12 = 1 + O(h^2) \quad \textrm{as} \quad h \rightarrow 0^+.
\end{align*}
Squaring both sides a final time, we conclude
\begin{align*}
\beta_*(h) = O(h^2) \quad \textrm{as} \quad h \rightarrow 0^+.
\end{align*}
In light of this result, we now posit the expansion $\beta_*(h) = \beta_2 h^2 + \mathscr{o}(h^2)$ as $h \rightarrow 0^+$ for some $\beta_2$ to be determined.  Substituting this expansion into \eqref{Fbetah0} and setting $F(\beta_*(h),h) = 0$, we find at leading order that $\beta_2 = 4/3$, proving (iii).  

In order to prove (iv), we claim that there exists $H_0 > 0$ such that $\beta_*'(h) > 0$ for all $h > H_0$. 
Indeed, we calculate 
\begin{align*}
\partial_h F(\beta,h) = \left(\frac{3\textrm{sech}^2(h)}{2\tanh^\frac12(h)}\right) g(\beta,h),
\end{align*}
where
\begin{align*}
g(\beta,h) &= 1 - \frac{(1+\beta)^\frac34}{3}\left( \frac{\textrm{sech}\left(h\left(1+\beta\right)^\frac12 \right)}{\textrm{sech}(h)} \right)^2 \left( \frac{\tanh(h)}{\tanh\left(h\left( 1+\beta\right)^\frac12 \right)} \right)^\frac12 \\ &\quad - \frac{(4+\beta)^\frac34}{3}\left( \frac{\textrm{sech}\left(h\left(4+\beta\right)^\frac12 \right)}{\textrm{sech}(h)} \right)^2 \left( \frac{\tanh(h)}{\tanh\left(h\left( 4+\beta\right)^\frac12 \right)} \right)^\frac12.
\end{align*}
From the monotonicity of $\tanh$, we have $\tanh(h)/\tanh(wh) < 1$ for $h> 0, \ w > 1$. Thus 
\begin{align*}
g(\beta,h) > 1 - \frac{(1+\beta)^\frac34}{3}\left( \frac{\textrm{sech}\left(h\left(1+\beta\right)^\frac12 \right)}{\textrm{sech}(h)} \right)^2 - \frac{(4+\beta)^\frac34}{3}\left( \frac{\textrm{sech}\left(h\left(4+\beta\right)^\frac12 \right)}{\textrm{sech}(h)} \right)^2.
\end{align*}
For $0 < \beta < 3$, we have the bounds
\begin{align*}
\frac{(1+\beta)^\frac34}{3}\left( \frac{\textrm{sech}\left(h\left(1+\beta\right)^\frac12 \right)}{\textrm{sech}(h)} \right)^2 < \frac{4^\frac34}{3} \quad \textrm{and} \quad \frac{(4+\beta)^\frac34}{3}\left( \frac{\textrm{sech}\left(h\left(4+\beta\right)^\frac12 \right)}{\textrm{sech}(h)} \right)^2 < \frac{7^\frac34}{3}\left(\frac{\textrm{sech}(2h)}{\textrm{sech}(h)} \right)^2,
\end{align*}
which follow from the monotonicty of $\textrm{sech}(\xi)$ for $\xi > 0$. It follows that
\begin{align*}
g(\beta,h) > 1 - \frac{4^\frac34}{3} - \frac{7^\frac34}{3}\left(\frac{\textrm{sech}(2h)}{\textrm{sech}(h)} \right)^2.
\end{align*}
Now $1 - 4^{\frac34}/3 > 0$ and  
$\textrm{sech}(2h)/\textrm{sech}(h) \rightarrow 0$ as $h \rightarrow \infty$, 
so that there exists $H_0 > 0$ such that 
\begin{align*}
1 - \frac{4^\frac34}{3} - \frac{7^\frac34}{3}\left(\frac{\textrm{sech}(2h)}{\textrm{sech}(h)} \right)^2 > 0 \quad \forall  h > H_0.
\end{align*}
Consequently, $\partial_h F(\beta,h)> 0$ for all $0 < \beta < 3$ and $h > H_0$. In particular, $(\partial_hF)(\beta_*(h),h) > 0$ for $h > H_0$.   
However, $\partial_\beta F(\beta,h) < 0$ for all $\beta > 0$ and $h > 0$, so that 
$\beta_*'(h) =  - \frac{(\partial_hF)(\beta_*(h),h)}{(\partial_\beta F)(\beta_*(h),h) }> 0$ for $h > H_0$, proving our claim. 

Because $\beta_*(h)$ is bounded above and increasing for $h > H_0$, we know that  $\beta_{*,\infty} := \lim_{h \rightarrow \infty} \beta_*(h)$ exists and is finite. Taking the limit of $F(\beta_*(h),h) = 0$ as $h \rightarrow \infty$, we arrive at the following equation for $\beta_{*,\infty}$:
\begin{align}
3 - (1+\beta_{*,\infty})^\frac14 - (4+\beta_{*,\infty})^\frac14 = 0. \label{betaStarInf}
\end{align}
The solution of \eqref{betaStarInf} is 
\begin{align*}
\beta_{*,\infty} = \frac{1}{48}\left[-687 + \left(70498161-7872768\sqrt{7}\right)^\frac13 + 3\left(2611043+291584\sqrt{7} \right)^\frac13 \right] = 2.7275...,
\end{align*}
consistent with the infinite depth results presented in \cite{CreNguStr}. %In particular, there exists $H_1>0$ sufficiently large such that $\beta_*(h)\in (2, 3)$ for $h>H_1$. 
%Using 
%\begin{align*} h(1+\beta_*(h))^\frac12 \rightarrow \infty \quad \textrm{and} \quad h(4+\beta_*(h))^\frac12 \rightarrow \infty \quad \textrm{as} \quad h \rightarrow \infty, \end{align*}  
%together with the expansion
From the expansion 
\begin{align*}
\tanh^\frac12(\xi) = 1 - e^{-2\xi} + O\left(e^{-4\xi} \right) \quad \textrm{as} \quad \xi \rightarrow \infty,
\end{align*}
%and the bound $2<\beta_*(h)<3$ for $h > H_1$ \Huy{this bound is not needed for the line below?},
 the expression \eqref{Fbetah} has the asymptotic behavior 
\begin{align}
F(\beta_*(h),h)  &= 3 - \left(1+\beta_*(h)\right)^\frac14 - \left(4+\beta_*(h)\right)^\frac14  - 3e^{-2h}  + \mathscr{o}\left(e^{-2h} \right) \quad \textrm{as} \quad h \rightarrow \infty. \label{Fbetahinf}
\end{align}
Define $\underline{\beta}(h) := \beta_*(h) - \beta_{*,\infty}$. Then  \eqref{Fbetahinf} becomes
\begin{align*}
F(\beta_*(h),h)  &= 3 - \left(1+\beta_{*,\infty} \right)^\frac14 \left(1+ \frac{\underline{\beta}(h)}{1+\beta_{*,\infty}}\right)^\frac14 - \left(4+\beta_{*,\infty} \right)^\frac14\left(1 + \frac{\underline{\beta}(h)}{4+\beta_{*,\infty}}\right)^\frac14  - 3e^{-2h} + \mathscr{o}\left( e^{-2h}\right),
\end{align*}
 as $h \rightarrow \infty$. Since  $\underline{\beta}(h) \rightarrow 0^-$ as $h \rightarrow \infty$,  combining  the expansion
$\left(1+\xi \right)^\frac14 = 1 + \frac14 \xi + O\left(\xi^2\right)$ with \eqref{betaStarInf} yields
\begin{align*}
F(\beta_*(h),h)) = -\frac{\underline{\beta}(h)}{4}\left(\left( 1+\beta_{*,\infty}\right)^{-\frac34} + \left( 4+\beta_{*,\infty}\right)^{-\frac34}  \right) - 3e^{-2h} + O\left(\underline{\beta}(h)^2 \right) + \mathscr{o}\left(e^{-2h} \right)  \quad \textrm{as} \quad h \rightarrow \infty.
\end{align*}
Because $F(\beta_*(h),h) = 0$, we obtain the following implicit formula for $\underline{\beta}(h)$:
\begin{align*}
\underline{\beta}(h) = \frac{-12e^{-2h}}{\left( 1+\beta_{*,\infty}\right)^{-\frac34} + \left( 4+\beta_{*,\infty}\right)^{-\frac34} }  + O\left(\underline{\beta}(h)^2 \right) + \mathscr{o}\left(e^{-2h} \right) \quad \textrm{as} \quad h \rightarrow \infty.  
\end{align*}
Hence $\underline{\beta}(h) = O\left(e^{-2h}\right)$ as $h \rightarrow \infty$. Therefore, we obtain 
\begin{align*}
\underline{\beta}(h) = \frac{-12e^{-2h}}{\left( 1+\beta_{*,\infty}\right)^{-\frac34} + \left( 4+\beta_{*,\infty}\right)^{-\frac34} } + \mathscr{o}\left(e^{-2h} \right) \quad \textrm{as} \quad h \rightarrow \infty, \label{betaUnderline}
\end{align*}
from which (iv) follows. 
\end{proof} 
\begin{figure}[tb]
\includegraphics[width=8cm]{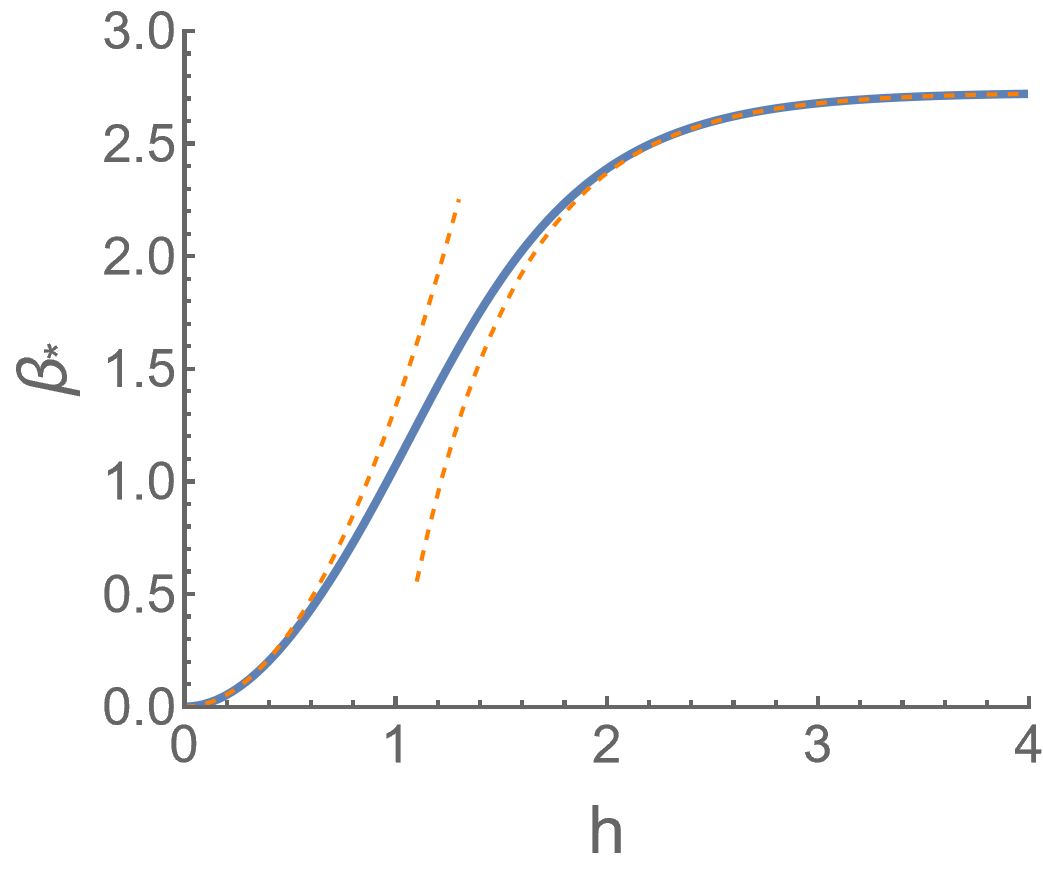}
\label{fig2}
\caption{A plot of $\beta_*$ as a function of $h$ (solid blue) along with its asymptotic expansions as $h \rightarrow 0^+$ and $h \rightarrow \infty$ (dashed orange) according to Proposition ~\ref{prop:rescond}.}
\end{figure}
\begin{lemm} \label{prop:lambda0Monotone}
 For each $\beta > 0$ and $h > 0$, we have %the functions $k\mapsto \Omega^0_{\pm}(k, \beta )$ are strictly increasing on the intervals $(-\infty, -1]$ and $[1, \infty)$.
\bq\label{dOmpm}
\frac{\partial}{\partial k} \textrm{Im}\left\{ \ld^0_\pm(k, \beta )\right\}>
\begin{cases} 
c_0\left(1-(1+\beta)^{-\mez}\right)\quad\text{if}~|k|\ge 1,\\
0\quad\text{if}~k\ne 0.
\end{cases}
\eq
\end{lemm}
\begin{proof}
For ease of notation, we define the auxiliary functions
\begin{align*}
 \omega(\xi,h) = \left(\xi\tanh(h\xi) \right)^\mez\quad\text{and}\quad \xi(k,\beta) = (k^2+\beta)^\mez, 
\end{align*}
so that  $ \textrm{Im}\left\{ \ld^0_\pm(k, \beta )\right\} = c_0 k \pm \omega(\xi(k,\beta),h)$. A direct calculation shows
\begin{align*}
\frac{\partial}{\partial k}\omega(\xi(k,\beta),h) = \frac{\sqrt{h}}{2}\left[\sqrt{\frac{\tanh(h\xi)}{h\xi}} + \text{sech}^{\frac32}(h\xi)\sqrt{\frac{h\xi}{\sinh(h\xi)}} \right] \frac{\partial}{\partial k} \xi(k,\beta)
\end{align*}
and $\left|\frac{\partial}{\partial k} \xi(k,\beta) \right| = \frac{|k|}{(k^2+\beta)^{\frac12}}>0$ for $k\ne 0$. 
%where
%\begin{align*}
%\frac{\partial}{\partial k} \xi(k,\beta) = \frac{k}{(k^2+\beta)^\mez}.
%\end{align*}
%We now show that $\left|\frac{\partial}{\partial k}\omega(\xi(k,\beta),h) \right|< c_0$ provided $|k| > 1$.
Using the inequalities $s/\sinh(s) < 1$ and $\text{sech}^{\frac32}(s) < \text{sech}(s)<\sqrt{\tanh(s)/s}$ for $s > 0$, we have
\begin{align*}
\left|\frac{\partial}{\partial k}\omega(\xi(k,\beta),h)\right| < \sqrt{h}\sqrt{\frac{\tanh(h\xi)}{h\xi}} \left|\frac{\partial}{\partial k} \xi(k,\beta) \right|\quad\forall k \ne 0.
\end{align*}
%If $|k|>1$, then $\xi > 1$ by definition. 
Since $\tanh(s)/s$ is decreasing for all $s > 0$,  we deduce 
\[
\left|\frac{\partial}{\partial k}\omega(\xi(k,\beta),h)\right|<\sqrt{h}\sqrt{\frac{\tanh(h)}{h}} \left|\frac{\partial}{\partial k} \xi(k,\beta) \right|=c_0 \left|\frac{\partial}{\partial k} \xi(k,\beta) \right|\quad\forall k\ne 0.
\]
Then, \eqref{dOmpm} follows from this and the inequalities
 \[
0< \left|\frac{\partial}{\partial k} \xi(k,\beta) \right| \le 
 \begin{cases}
(1+\beta)^{-\mez}\quad\text{if}~|k|\ge 1,\\
  1\quad\text{if}~k\ne  0.
 \end{cases}
 \]
\end{proof}
Given $\beta_*$, we define $\sigma$ by 
\bq\label{def:sigma}
i\sigma:=\ld^0_+(-2, \beta_*,h)=\ld^0_-(1, \beta_*,h).
\eq
The second inequality in \eqref{dOmpm} implies that the functions $k\mapsto  \textrm{Im}\left\{ \ld^0_\pm(k, \beta )\right\}$ are strictly increasing on $\Rr$, and hence 
\bq\label{Delta0}
\Delta_0(i\sigma; k, \beta _*)=0\iff k\in \{1, -2\}.
\eq
Thus $i\sigma$ is a {\it double eigenvalue} of $\cL_{0,\beta_*}$ for all $h > 0$. Moreover, the first inequality in \eqref{dOmpm} implies that $i\sigma$ is {\it isolated} in the spectrum of  $\cL_{0,\beta_*}$. The  eigenspace associated  to $i \sigma$ is 
\bq     \label{basisU}
%\cU:=
N(\cL_{0,\beta_*}-i\sigma\I)=\text{span}\{U_{1}, U_{2}\},
\eq
where $\I$ denotes the identity operator and 
\begin{align}      \label{def:U}
U_{1}(x;h)&=\begin{bmatrix} i(1+\beta_* )^\frac14\tanh^\mez\left(h(1+\beta_*)^\mez \right)\\ 1 \end{bmatrix}e^{ix}, \\ U_{2}(x;h) &=\begin{bmatrix} -i(4+\beta_*)^\frac14\tanh^\mez\left(h(4+\beta_*)^\mez \right) \\ 1 \end{bmatrix}e^{-2ix}.
\end{align}
%Thus $\cL_{0,\beta}U_j=i\sigma U_j$.  
For ease of notation, we define 
  \bq\label{def:gammaj}
 \g_{j} := (j^2 + \beta_*)^\frac14\tanh^\mez\left(h\left(j^2+\beta_*\right)^\mez \right),
  \eq
  so that 
  \bq\label{sigma:gamma12}
  \sigma=c_0-\g_{1}=-2c_0+\g_{2}.
  \eq 
 
%%%%%%%%%%%%%%%%%%%%%%
\subsection{Kato's  perturbed basis} 
%Recall from \eqref{def:sigma} or \eqref{sigma:gamma12} that there is a double eigenvalue at $i\sigma$.  
We fix any depth $h\in (0, \infty)$ and let $i\sigma$ be the isolated double eigenvalue \eqref{def:sigma} of $\cL_{0, \beta_*}$. Then we fix  $\Gamma$ a  circle centered at  $i\sigma$ with sufficiently small radius so that $\Gamma\subset \rho(\cL_{\eps,\beta})$ for $(\eps, \beta)$ close to $(0 ,\beta_*)$. We denote the spectral projection associated to $\Gamma$ by
\bq\label{def:Peps}
P_{\eps, \delta} = -\frac1{2\pi i} \int_\Gamma (\cL_{\eps, \beta_*+\delta} - \lb)^{-1} d\lb\in \cL\big((L^2(\T))^2, (H^1(\T))^2\big).  \eq 
The analyticity \eqref{analyticity:cL} of $\cL_{\eps,\beta}$ implies that $P_{\eps, \delta}$ is analytic in $(\eps, \delta)$ near $(0, 0)$.
 Recall from \eqref{basisU} that the set $\{U_1, U_2\}$ forms a basis of the range $\cV_{0, 0}:=R(P_{0,0})$.  
Our goal is to perturb the basis in a convenient way.  Following Kato \cite{Kato} and Berti {\it et al} \cite{Berti1}, we define perturbing transformation operators $\cK_{\eps, \delta}$ as 
\bq \label{def:Kato}
\cK_{\eps, \delta} := \{1-(P_{\eps, \delta}-P_{0,0})^2\}^{-\mez} \{P_{\eps, \delta} P_{0,0} + (1-P_{\eps, \delta})(1-P_{0,0})\}   \eq
and a {\it perturbed basis} $\{U^{\eps,\delta}_1, U^{\eps,\delta}_2\}$ by 
\bq
U^{\eps, \delta}_m=\cK_{\eps, \delta} U_m \quad (m=1,2).
\eq
Since $P_{\eps, \delta}$ is analytic in $(\eps, \delta)$, so are $\cK_{\eps, \delta}$ and $U^{\eps, \delta}_m$. Moreover, when acting on $U_m$, the last term from \eqref{def:Kato} 
in $\cK_{\eps, \delta} U_m$ vanishes since $(1-P_{0,0})U_m=0$.  By virtue of  Lemma 3.1 in \cite {Berti1}, the perturbed range $\cV_{\eps, \delta}:=R(P_{\eps, \delta})$ is spanned by $\{U^{\eps,\delta}_1, U^{\eps,\delta}_2\}$, and we have 
\[
\cL_{\eps, \beta_*+\delta}: \cV_{\eps, \delta}\to \cV_{\eps, \delta}.
\]
%the perturbed basis spans the eigenspace of eigenvalues of $\cL_{\eps, \beta_*+\delta}$ that bifurcate from $i\sigma$ for sufficiently small $\varepsilon$ and $\delta$.
%Now we expand $\cK_\eps U_j$ in its Taylor series in powers of $\eps$, as follows.  
%\bq 
%\cK_\eps U_j  =  \left\{ 1 + \eps P_0'  + \tfrac12 \eps^2 (P_0'P_0'+P_0'') + O(\eps^3) \right\} U_j  .   \eq
%We computed $P_0'U_j$ and $P_0''U_j$ above.  
%In order to get $\cK U_j$, we still have to compute $P_0' P_0' U_j$. 
 %That's a lot of terms!!!  It's much more complicated than what Berti did in his Lemma A.3.  
 By direct calculation, we obtain
\bq\label{JUU:0}
(JU_1, U_2)=(JU_2, U_1)=0,\quad (JU_1, U_1)=-i4\pi \g_1,\quad (JU_2, U_2)=i4\pi \g_2.
\eq
Since  $\cK_{\eps, \delta}$ is symplectic, {\it i.e.} $\cK_{\eps, \delta}^*J\cK_{\eps, \delta}=J$ (see Lemma 3.2 in \cite{Berti1}), we have $(J\cK_{\eps, \delta}U, \cK_{\eps, \delta} V)=(JU, V)$. Consequently,  \eqref{JUU:0} yields
\bq\label{JUU}
(JU_1^{\eps, \delta}, U_2^{\eps, \delta})=(JU_2^{\eps, \delta}, U_1^{\eps, \delta})=0,\quad (JU_1^{\eps, \delta}, U_1^{\eps, \delta})=-i4\pi \g_1,\quad (JU_2^{\eps, \delta}, U_2^{\eps, \delta})=i4\pi \g_2.
\eq
 In view of \eqref{JUU:0} and \eqref{JUU},  we  normalize the eigenvectors according to 
\bq
V_m=\frac{1}{\sqrt{\gamma_m}}U_m,\quad V_m^{\eps, \delta}=\frac{1}{\sqrt{\gamma_m}}U^{\eps, \delta}_m.
\eq
%According to Lemma 3.1 in \cite {Berti1}, the range $\cV_{\eps, \delta}$ of $P_{\eps, \delta}$ 
%is 
%$\cV_{\eps, \delta}=\cK_{\eps, \beta_*+\delta} \cV_{0, 0}$, 
%so that $\{U_1^{\eps, \delta}, U_2^{\eps, \delta}\}$ is indeed a basis of $\cV_{\eps, \delta}$. 
%%%%%%%%%%%
\begin{lemm}
Using our notation  $\beta=\beta_*+\delta$, the {$2\times 2$ matrix}  
that represents the linear operator $\cL_{\eps,\beta}=J\cH_{\eps,\beta}:\cV_{\eps, \delta}\to \cV_{\eps, \delta}$ with respect to the basis $\{V_1^{\eps, \delta}, V_2^{\eps, \delta}\}$ is 
\bq\label{matrixL} 
\textrm{L}_{\eps,\delta}=\begin{bmatrix} -\frac{i}{4\pi }(\cH_{\eps,\beta} V_1^{\eps, \delta}, V_1^{\eps, \delta}) &  \frac{i}{4\pi }(\cH_{\eps,\beta} V_1^{\eps, \delta}, V_2^{\eps, \delta})   \\
- \frac{i}{4\pi }(\cH_{\eps,\beta} V_2^{\eps, \delta}, V_1^{\eps, \delta}) & \frac{i}{4\pi }(\cH_{\eps,\beta} U_2^{\eps, \delta}, V_2^{\eps, \delta}) 
\end{bmatrix}. 
\eq 
$\textrm{L}_{\eps, \delta}$ is purely imaginary and
\bq\label{L:reversediagonal}
(\textrm{L}_{\eps, \delta})_{12}=-(\textrm{L}_{\eps, \delta})_{21}.
\eq
\end{lemm}
\begin{proof}
The claim \eqref{matrixL}  follows from \eqref{JUU} and the Hamiltonian form $\cL_{\eps, \beta}=J \cH_{\eps, \beta}$. We refer to the proof of Lemma 3.1 in \cite{CreNguStr} for details. The second claim follows from the property \eqref{cG:even} of $\cG_{\eps, \beta}$ and the fact that the functions $p, q, \zeta'$ are real and even; see the end of Section 3.2 in \cite{CreNguStr}. 
\end{proof} 
We have reduced the original  transverse instability problem to the spectral analysis of the $2\times 2$ matrix $\textrm{L}_{\eps, \delta}$. To capture the unstable eigenvalues of $\textrm{L}_{\eps, \delta}$, we shall expand $\textrm{L}_{\eps, \delta}$ to third order in $\eps$ and $\delta$ in the next section.  
%%%%%%%%%%%%%%%%%%%%%%% SECTION 4  %%%%%%%%%%%%%
\section{Third-order expansions}% Expansions of $\cL_{\eps,\beta}$ up to third order}
\subsection{ Third-order expansions of $\cL_{\eps,\beta}$}\label{Sec:R}
 We  recall the Hamiltonian operator \eqref{def:Hamiltonian}
 \[
 \cH_{\eps,\beta}=\begin{bmatrix} \frac{1 +q_{}(x)}{\zeta_{}'(x)} & -p_{}(x)\p_x \\ \p_x(p_{}(x)\cdot) & \cG_{\eps,\beta}\end{bmatrix}.
 \]
By virtue of Proposition \ref{prop:expandpq} and  Theorem \ref{theo:flattenG} (ii), for each $h>0$, $\cH_{\eps,\beta}$ is analytic in $(\eps, \beta)\in (-\eps_0(1), \eps_0(1))\times (0, \infty)$ with values in $\mathcal{L}(H^1(\T), L^2(\T))$. In particular, for fixed $\beta>0$, $\cG_{\eps,\beta}$ can be expanded in power series of $\eps$ as
 \bq\label{expandcG}
 \cG_{\eps,\beta}=\sum_{j=0}^\infty \eps^j R_{j,\beta},\quad |\eps|<\eps_0(1).
 \eq
 Then, invoking  the expansions in Proposition  \ref{prop:expandpq}, we obtain the third-order expansion in $\eps$ of $\cH_{\eps, \beta}$:
 \bq    \label{expandHepsbeta}
\cH_{\eps,\beta}=\sum_{j=0}^3\eps^j\cH^{j}_{\beta}+O(\eps^4),
\eq
			where 
\bq
\begin{aligned} \label{Hexpansion:eps}
&\cH^{0}_{\beta}=\begin{bmatrix} 
1 &-c_0\p_x \\ c_0\p_x & R_{0,\beta}
\end{bmatrix},\quad c_0=\sqrt{\tanh(h)}, \\
&\cH^{1}_{\beta}=\begin{bmatrix} 
r_{1,1}\cos x& -p_{1,1} \cos x\p_x\\ \p_x(p_{1,1}\cos x\cdot)  &   R_{1,\beta}
\end{bmatrix},\\
& \cH^{2}_{\beta}=\begin{bmatrix} 
r_{2,0}+r_{2,2}\cos(2x)& -\big[p_{2,0}+p_{2,2}\cos(2x)\big]\p_x\\ \p_x\{[p_{2,0}+p_{2,2}\cos(2x)]\cdot)\} & R_{2,\beta}
\end{bmatrix},\\
&\cH^{3}_{\beta}=\begin{bmatrix}
r_{3,1}\cos x+r_{3,3}\cos(3x)& -\big[p_{3,1}\cos x+p_{3,3}\cos(3x)\big]\p_x\\ \p_x\left\{\big[p_{3,1}\cos x+p_{3,3}\cos(3x)\big]\cdot\right\} & R_{3,\beta} 
\end{bmatrix}.
\end{aligned}
\eq
  The coefficients $r_{1,1},\dots, p_{3,3}$ are given in Prop. \ref{prop:expandpq}.  
%\bq
%\wh{R_3v}(k)=E^-_k\wh{f_2}(k-3)+E^+_k\wh{f_2}(k+3)+\text{coeff.} \wh{f_2}(k\pm 1)
%\eq
%\begin{rema}
% { Acting on the basis \eqref{def:U}, we have 
 % $\cH^{0}_{\beta_*,h} U_1=-i\sigma \begin{bmatrix}1\\ -i\g_1 \end{bmatrix} e^{ix}$ 
 % and $\cH^{0}_{\beta_*,h} U_2=-i\sigma \begin{bmatrix}1\\ i\g_2 \end{bmatrix} e^{-2ix}$. } 
  %Maybe this should be a remark? 
%\end{rema}
 The nonlocal operators $R_{j,\beta}$ are Fourier multipliers.   
 We explicitly compute them for $0 \leq j \leq 3$ in the following proposition.  
 %%%%%%%%%%%%%%
\begin{prop}\label{prop:Rj}
The Fourier multipliers $R_j$ for $0 \leq j \leq 3$ take the form
\begin{align} \label{form:R0} &\wh{R_{0,\beta}f}(k)=\textcolor{black}{A^{0}_{k,\beta}}\wh{f}(k)=\sqrt{|k|^2+\beta}\ \tanh(h\sqrt{|k|^2+\beta})\wh{f}(k),\\\label{form:R1}
&\wh{R_{1,\beta}f}(k)=\textcolor{black}{B^{-1}_{k,\beta}}\wh{f}(k-1)+\textcolor{black}{B^{1}_{k,\beta}}\wh{f}(k+1),\\\label{form:R2}
&\wh{R_{2,\beta}f}(k)=C^{-2}_{k,\beta}\wh{f}(k-2)+C^0_{k,\beta}\wh{f}(k)+C^{2}_{k,\beta}\wh{f}(k+2),\\ \label{form:R3}
&\wh{R_{3,\beta}f}(k)= D^{-3}_{k,\beta}\wh{f}(k-3)+ D^{-1}_{k,\beta}\wh{f}(k-1) + D^{1}_{k,\beta}\wh{f}(k+1)+D^{3}_{k,\beta}\wh{f}(k+3), 
\end{align}
where the coefficients $A^s_{k,\beta}$, $B^s_{k,\beta}$, $C^s_{k,\beta}$, and $D^s_{k,\beta}$ are explicit functions of $k$, $\beta$, and $h$ derived below.
\end{prop}
%%%%%%%%%%
\begin{proof}

We write the expansions 
\[
\zeta=x+\eps\zeta^1+\eps^2\zeta^2+\eps^3\zeta^3+O(\eps^4),\quad \eta^*=\eps\eta^1+\eps^2\eta^2+\eps^3\eta^3+O(\eps^4),   
\]
where the coefficients are given in Prop. \ref{prop:expandpq} and Prop. \ref{eta expansion}.  
A Taylor expansion leads to 
\bq\label{compose:zeta}
\begin{aligned}
\eta^*(\zeta(x))&=\eps \eta^1(x)+\eps^2(\zeta^1\p_x\eta^1(x)+\eta^2(x))\\
&\quad+\eps^3\left\{\zeta^2(x)\p_x\eta^1(x)+\mez(\zeta^1(x))^2\p^2_x\eta^1(x)+\zeta^1(x)\p_x\eta^2(x)+\eta^3(x)\right\}+O(\eps^4).
\end{aligned}
\eq
Thereby we obtain 
\begin{align}
\eta^*(\zeta_{}(x))&
=\eps \cos x+\frac12 \eps^2\big[2\eta_{2,0}-\zeta_{1,1} + (\zeta_{1,1}+2\eta_{2,2})\cos(2x) \big]+\frac18\eps^3  \Big[   \big(-\zeta_{1,1}^2-4\zeta_{2,2}-8\zeta_{1,1}\eta_{2,2} \nonumber \\
&
\quad +8\eta_{3,1} \big)\cos x + \big(\zeta_{1,1}^2 +4\zeta_{2,2}+8\zeta_{1,1}\eta_{2,2}+8\eta_{3,3}\big)\cos(3x)  \Big]  + O(\eps^4).  
\end{align} 
Taking the Fourier transform of this expression, we use the simple fact that 
$\widehat {\cos mx}(k) = \pi \delta(k-m)+\pi\delta(k+m)$.  
Inserting this expansion 
into the Fourier expansion of the Riemann stretch \eqref{z1}, we find that  
\begin{align*}
\p_x X&= 1 + \varepsilon \frac{C_{h+z}}{C_{h}}\zeta_{1,1}\cos x +2\eps^2\frac{C_{2(h+z)}}{C_{2h}}\zeta_{2,2}\cos(2x) + \eps^3\biggr[ \Big(\frac{d_2S_z}{C_h^2}\zeta_{1,1} + \frac{C_{h+z}}{C_h}\zeta_{3,1}\Big)\cos x \\ &\quad+3\frac{C_{3(h+z)}}{C_{3h}}\zeta_{3,3}\cos(3x) \biggr] + O(\eps^4), \\
 \p_zX&=\varepsilon \frac{S_{h+z}}{C_{h}}\zeta_{1,1}\sin x +2\eps^2\frac{S_{2(h+z)}}{C_{2h}}\zeta_{2,2}\sin(2x) + \eps^3\biggr[ \Big(\frac{d_2C_z}{C_h^2}\zeta_{1,1} + \frac{S_{h+z}}{C_h}\zeta_{3,1}\Big)\sin x \\ &\quad+3\frac{S_{3(h+z)}}{C_{3h}}\zeta_{3,3}\sin(3x) \biggr] + O(\eps^4),
\end{align*}
where we have defined   
$C_{s} := \cosh(s) \quad \textrm{and} \quad S_{s} := \sinh(s).  $
We have also used the expansion of $h_{\eps}$ in \eqref{heps expansion} and substituted 
$\sinh(mh) = \cosh(mh) \tanh(mh)$  in the denominators,as well as expansions of $\tanh(mh)$ 
and standard trigonometric identities.  
It follows that the Jacobian is 
\begin{align*}      
\cJ&=1+2\eps\frac{C_{h+z}}{C_h}\zeta_{1,1}\cos x + \eps^2\biggr[ \frac{C_{2(h+z)}}{2C_h^2}\zeta_{1,1}^2 + \Big(4\frac{C_{2(h+z)}}{C_{2h}}\zeta_{2,2}+\frac{1}{2C_h^2}\zeta_{1,1}^2\Big)\cos(2x)\biggr] \\ &\quad+2\eps^3\biggr[ \Big(\frac{d_2S_z}{C_h^2}\zeta_{1,1}+ \frac{C_{3(h+z)}}{C_hC_{2h}}\zeta_{1,1}\zeta_{2,2} + \frac{C_{h+z}}{C_h}\zeta_{3,1}\Big)\cos x + \Big(\frac{C_{h+z}}{C_hC_{2h}}\zeta_{1,1}\zeta_{2,2}+3\frac{C_{3(h+z)}}{C_{3h}}\zeta_{3,3} \Big)\cos(3x) \biggr] \\ &\quad+O(\eps^4). 
\end{align*}
%and its reciprocal is 
%\[
% \cJ^{-1}=1-2\eps e^z \cos x +\eps^2e^{2z}(1-2\cos(2x))+\eps^3\left\{e^{3z}[2\cos x-3\cos(3x)]+2e^z\cos x\right\}+O(\eps^4).
% \]

 Given these various expansions, we return to $\Theta=\Theta_f$ which satisfies  
\bq
\begin{cases}
\Delta_{x, z}\Tt-\beta\cJ\Tt=0\quad\text{in } \{(x, z): -h_\eps < z<0\},\\
\Tt(x, 0)=f(x), \quad     %:=(f_0\circ\zeta)(x),\\
\partial_z\Tt(x,-h_\eps) =  0,
\end{cases}
\eq
as defined in \eqref{system:Theta}. 
We recall from the proof of  Theorem \ref{theo:flattenG} that $\cG_{\eps,\beta}f=\p_z\Tt(\cdot, 0)$ and 
that $\Tt$ is analytic in $(\varepsilon, \beta)$ for each $h > 0$. In particular, for fixed $\beta>0$ and $h > 0$,  we can expand $\Tt=\Tt^0+\eps \Tt^1+\eps^2\Tt^2+O(\eps^3)$, where $\Tt^0$, $\Tt^1$, $\Tt^2$, and $\Tt^3$ satisfy, by use of the expansion of the Jacobian, the equations 
\bq\label{sys:Tt0}
\begin{cases}
\Delta_{x, z}\Tt^0-\beta\Tt^0=0\quad\text{in } \{(x, z): -h<z<0\},\\
\Tt^0(x, 0)=f(x),\\
\p_z\Tt^0(x,-h)= 0,
\end{cases}
\eq
\bq\label{sys:Tt1}
\begin{cases}
\Delta_{x, z}\Tt^1-\beta\Tt^1= 2\beta\frac{C_{h+z}}{C_h}\zeta_{1,1}(\cos x) \Tt^0\quad\text{in } \{(x, z): -h<z<0\},\\
\Tt^1(x, 0)=0,\\
\p_z\Tt^1(x,-h)=0,
\end{cases}
\eq
\bq\label{sys:Tt2}
\begin{cases}\begin{aligned}
&\Delta_{x, z}\Tt^2-\beta\Tt^2=\beta\biggr[ \frac{C_{2(h+z)}}{2C_h^2}\zeta_{1,1}^2 + \Big(4\frac{C_{2(h+z)}}{C_{2h}}\zeta_{2,2}+\frac{1}{2C_h^2}\zeta_{1,1}^2\Big)\cos(2x)\biggr]\Tt^0 \\ &\hspace{3.1cm} +  2\beta\frac{C_{h+z}}{C_h}\zeta_{1,1}\cos x \Tt^1 \quad \text{in } \{(x, z): -h<z<0\},\\
&\Tt^2(x, 0)=0,\\
&\p_z\Tt^2(x,-h) = h_2\partial_z^2\Tt^0(x,-h),
\end{aligned}
\end{cases}
\eq
\bq\label{sys:Tt3}
\begin{cases}
\begin{aligned}\Delta_{x, z}\Tt^3-\beta\Tt^3&= 2\beta\biggr[ \Big(\frac{d_2S_z}{C_h^2}\zeta_{1,1}+ \frac{C_{3(h+z)}}{C_hC_{2h}}\zeta_{1,1}\zeta_{2,2} + \frac{C_{h+z}}{C_h}\zeta_{3,1}\Big)\cos x + \Big(\frac{C_{h+z}}{C_hC_{2h}}\zeta_{1,1}\zeta_{2,2} \\&\hspace{0.85cm}+3\frac{C_{3(h+z)}}{C_{3h}}\zeta_{3,3} \Big)\cos(3x) \biggr]\Tt^0 + \beta\biggr[ \frac{C_{2(h+z)}}{2C_h^2}\zeta_{1,1}^2 + \Big(4\frac{C_{2(h+z)}}{C_{2h}}\zeta_{2,2} \\&\hspace{0.85cm}+\frac{1}{2C_h^2}\zeta_{1,1}^2\Big)\cos(2x)\biggr]\Tt^1 +2\beta\frac{C_{h+z}}{C_h}\zeta_{1,1}\cos x \Tt^2  \quad \text{in } \{(x, z): -h<z<0\},
\end{aligned}\\
\Tt^3(x, 0)=0,\\
\p_z\Tt^3(x,-h) = h_2\p_z^2\Tt^1(x,-h).
\end{cases}
\eq
Comparing with the expansion \eqref{expandcG}, we find that $R_{j,\beta}f=\p_z\Tt^j(\cdot, 0)$. 

We consider the functions $\Theta^j$, one at a time.  For $j=0$, we take the Fourier transform of \eqref{sys:Tt0} in $x$, leading to the new boundary value problem
\bq\label{sys:Tt0:v2}
\begin{cases}
\partial_z^2\wh{\Tt^0}-(k^2+\beta)\wh{\Tt^0}=0\quad\text{in } \{(k, z): -h<z<0\},\\
\wh{\Tt^0}(x, 0)=\wh{f}(k),\\
\p_z\wh{\Tt^0}(k,-h)= 0,
\end{cases}
\eq
For ease of notation, we define the functions
\begin{align}
\mathcal{C}_{i,j}(z) := \cosh\left\{z\left[i+\left((j+k)^2+\beta \right)^\mez \right] \right\}, \\ \mathcal{S}_{i,j}(z) := \sinh\left\{z\left[i+\left((j+k)^2+\beta \right)^\mez \right]\right\}.
\end{align} 
The general solution of \eqref{sys:Tt0:v2} may be expressed in terms of these functions as
\begin{align}
\wh{\Tt^0}(k,z) = \mathcal{A}^{(0)}_1\mathcal{C}_{0,0}(z) + \mathcal{A}^{(0)}_{2}\mathcal{S}_{0,0}(z),
\end{align}
where $\mathcal{A}^{(0)}_1$ and $\mathcal{A}^{(0)}_2$ are constants. Enforcing the requisite boundary conditions on $\wh{\Tt^0}$, we find
\begin{align}
\mathcal{A}^{(0)}_1 &= 1, \\
\mathcal{A}^{(0)}_2 &=  \tanh\left(h(k^2+\beta)^\mez\right),
\end{align}
so that
\bq\label{Tt0}
\wh{\Tt^0}(k, z)=\Big[\mathcal{C}_{0,0}(z)+ \tanh\left(h(k^2+\beta)^\mez\right)\mathcal{S}_{0,0}(z) \Big]\wh{f}(k).
\eq
Consequently,
\bq
\partial_z\wh{\Tt^0}(k,0) = \wh{R_{0,\beta}f}(k)=A^0_{k,\beta}\wh{f}(k), \quad \text{where} \quad A^0_{k,\beta} := (k^2+\beta)^\mez\tanh(h(k^2+\beta)^\mez)  
\eq 
with the notation \eqref{expandcG}.  
In other words, $R_{0,\beta}$ is the  Fourier multiplier operator, 
\bq
R_{0,\beta}f=(|D|^2+\beta)^\mez\tanh(h(|D|^2+\beta)^\mez)f, \quad D = -i\partial_x,
\eq
which is equivalent to $\cG_{0, \beta}f$ given in \eqref{form:G0}.
\begin{rema}
    If $\beta = 0$, we find $R_{0,0} = |D|\tanh(h|D|)$, in agreement with \cite{Berti2}. 
    If instead $\beta \neq 0$ and we let $h \rightarrow \infty$, we find $R_{0,\beta} \rightarrow (|D|^2+\beta)^\mez$, in agreement with \cite{CreNguStr}.
\end{rema}
We can follow a similar series of calculations to determine $R_{1,\beta}$. The Fourier transform in $x$ of the equation for $\Tt^1$ given in \eqref{sys:Tt1} is 
 \bq\label{ODE:g}
 \begin{aligned}
\p_z^2\wh{\Tt^1}-(k^2+\beta)\wh{\Tt^1}&= \beta \frac{C_{h+z}}{C_h}\zeta_{1,1}\Big[\wh{\Tt^0}(-1+k,z) + \wh{\Tt^0}(1+k,z) \Big].
\end{aligned}
\eq
Substituting \eqref{Tt0} into the right side of \eqref{ODE:g} and expanding, the equation for $\Theta^1$ simplifies to
\begin{align}
\p_z^2\wh{\Tt^1}-(k^2+\beta)\wh{\Tt^1} &= \sum_{i,j \in \{-1,1\}}\Big(\mathfrak{c}^{(1)}_{i,j}\mathcal{C}_{i,j}(z) + \mathfrak{s}^{(1)}_{i,j}\mathcal{S}_{i,j}(z)\Big),
\end{align}
where
\allowdisplaybreaks
\begin{align}
\mathfrak{c}^{(1)}_{-1,-1} &= -\frac12 \beta\zeta_{1,1}\Big\{-1 + \tanh(h)\tanh\left[h((-1+k)^2+\beta)^\mez\right]\Big\}\wh{f}(-1+k), \\
\mathfrak{c}^{(1)}_{1,-1} &= \frac12 \beta\zeta_{1,1}\Big\{1 + \tanh(h)\tanh\left[h((-1+k)^2+\beta)^\mez\right]\Big\}\wh{f}(-1+k), \\
\mathfrak{c}^{(1)}_{-1,1} &= -\frac12 \beta\zeta_{1,1}\Big\{-1 + \tanh(h)\tanh\left[h((1+k)^2+\beta)^\mez\right]\Big\}\wh{f}(1+k), \\
\mathfrak{c}^{(1)}_{1,1} &= \frac12 \beta\zeta_{1,1}\Big\{1 + \tanh(h)\tanh\left[h((1+k)^2+\beta)^\mez\right]\Big\}\wh{f}(1+k), \\
\mathfrak{s}^{(1)}_{-1,-1} &= \frac12 \beta\zeta_{1,1}\Big\{-\tanh(h)+\tanh\left[h((-1+k)^2+\beta)^\mez \right]\Big\}\wh{f}(-1+k), \\
\mathfrak{s}^{(1)}_{1,-1} &= \frac12 \beta\zeta_{1,1}\Big\{\tanh(h)+\tanh\left[h((-1+k)^2+\beta)^\mez \right]\Big\}\wh{f}(-1+k), \\
\mathfrak{s}^{(1)}_{-1,1} &= \frac12 \beta\zeta_{1,1}\Big\{-\tanh(h)+\tanh\left[h((1+k)^2+\beta)^\mez \right]\Big\}\wh{f}(1+k), \\
\mathfrak{s}^{(1)}_{1,1} &= \frac12 \beta\zeta_{1,1}\Big\{\tanh(h)+\tanh\left[h((1+k)^2+\beta)^\mez \right]\Big\}\wh{f}(1+k).
\end{align}
The general solution of \eqref{ODE:g} is
\begin{equation}       \label{Theta1}
\wh{\Tt^1}(k,z) = \mathcal{A}^{(1)}_1 \mathcal{C}_{0,0}(z) + \mathcal{A}^{(1)}_2\mathcal{S}_{0,0}(z) + \sum_{i,j\in\{-1,1\}}\Big( \tilde{\mathfrak{c}}^{(1)}_{i,j}\mathcal{C}_{i,j}(z) +\tilde{\mathfrak{s}}^{(1)}_{i,j}\mathcal{S}_{i,j}(z) \Big),
\end{equation}
where $\mathcal{A}^{(1)}_1$ and $\mathcal{A}^{(1)}_2$ are arbitrary constants and 
\begin{align}
\tilde{\mathfrak{c}}^{(1)}_{i,j} = \frac{\mathfrak{c}^{(1)}_{i,j}}{\Big[i + \left( (j+k)^2 + \beta \right)^\mez \Big]^2 - (k^2+\beta)}, \\
\tilde{\mathfrak{s}}^{(1)}_{i,j} = \frac{\mathfrak{s}^{(1)}_{i,j}}{\Big[i + \left( (j+k)^2 + \beta \right)^\mez \Big]^2 - (k^2+\beta)}.
\end{align}
%\[
%g(k, z)=C_1e^{\Om(k)z}+C_2e^{-\Om(k)z}+g_*(k, z),
%\]
%where the particular solution $g_*(k, z)$ can be found via the method of variation of constants. Namely, 
%\begin{align*}
%g_*&=\frac{1}{2\Om(k)}e^{\Om(k)z}\int e^{-\Om(k)z}\left\{-\beta e^z\wh{f}(k-1)e^{\Om(k-1)z}-\beta e^z\wh{f}(k+1)e^{\Om(k+1)z} \right\}dz\\
%&\quad-\frac{1}{2\Om(k)}e^{-\Om(k)z}\int e^{\Om(k)z}\left\{-\beta e^z\wh{f}(k-1)e^{\Om(k-1)z}-\beta e^z\wh{f}(k+1)e^{\Om(k+1)z} \right\}dz.
%\end{align*}
Enforcing the requisite boundary conditions on $\wh{\Tt^1}$, we determine
\begin{align}
\mathcal{A}^{(1)}_1 &= -\sum_{i,j\in \{-1,1\}} \tilde{\mathfrak{c}}^{(1)}_{i,j}, \\
\mathcal{A}^{(1)}_2 &=  -\sum_{i,j\in \{-1,1\}} \biggr\{\Big[ \frac{i+\left((j+k)^2+\beta \right)^\mez}{\mathcal{C}_{0,0}(h)(k^2+\beta)^\mez}\Big]\Big[ -\tilde{\mathfrak{c}}^{(1)}_{i,j}\mathcal{S}_{i,j}(h) + \tilde{\mathfrak{s}}^{(1)}_{i,j}\mathcal{C}_{i,j}(h)\Big] + \tilde{\mathfrak{c}}^{(1)}_{i,j}\tanh\left(h\left( k^2+\beta \right)^\mez\right) \biggr\}.
\end{align}
 %Therefore, in order to have  $\p_z g\to 0$ as $z\to -\infty$, $g$ must have the form 
%\[
%g(k, z)=C_1e^{\Om(k)z}+g_*(k, z).
%\]
A direct calculation from \eqref{sys:Tt0} yields 
\begin{align} \label{dzTt1}
\partial_z\wh{\Tt^1}(0,z) = \mathcal{A}_2^{(1)}(k^2+\beta)^\mez + \sum_{i,j\in\{-1,1\}} \tilde{\mathfrak{s}}^{(1)}_{i,j}\left(i + \left((j+k)^2+\beta \right)^\mez \right).
\end{align}
Substituting $\mathcal{A}^{(1)}_2$, $\tilde{\mathfrak{s}}^{(1)}_{i,j}$, and $\tilde{\mathfrak{c}}^{(1)}_{i,j}$ into \eqref{dzTt1} and expanding, we find that 
\begin{align}
\partial_z\wh{\Tt^1}(0,z) = B_{k,\beta}^{-1}\wh{f}(-1+k) + B_{k,\beta}^1\wh{f}(1+k),
\end{align}
where
\allowdisplaybreaks
\begin{align}
B_{k,\beta}^{-1} &:= \frac{1}{2} \biggr\{\beta -\sqrt{\beta +(k-1)^2} \sqrt{\beta +k^2} \tanh \left(h \sqrt{\beta +(k-1)^2}\right) \tanh \left(h \sqrt{\beta +k^2}\right) \nonumber \\ &\quad +\coth (h) \left[k \sqrt{\beta +(k-1)^2} \tanh \left(h \sqrt{\beta +(k-1)^2}\right)-(k-1) \sqrt{\beta +k^2} \tanh \left(h \sqrt{\beta +k^2}\right)\right] \nonumber \\ &\quad +k^2-k\biggr\}, \label{Bneg1} \\
B_{k,\beta}^{1} &:= \frac{1}{2} \biggr\{\beta -\sqrt{\beta +k^2} \sqrt{\beta +(k+1)^2} \tanh \left(h \sqrt{\beta +k^2}\right) \tanh \left(h \sqrt{\beta +(k+1)^2}\right) \nonumber \\&\quad +\coth (h) \left[(k+1) \sqrt{\beta +k^2} \tanh \left(h \sqrt{\beta +k^2}\right)-k \sqrt{\beta +(k+1)^2} \tanh \left(h \sqrt{\beta +(k+1)^2}\right)\right] \nonumber \\ &\quad +k^2+k\biggr\}. \label{Bpos1}
\end{align}
It follows that $R_{1,\beta}$ is the operator given by 
\begin{align}
R_{1,\beta} =  B_{k,\beta}^{-1}\wh{f}(-1+k) + B_{k,\beta}^1\wh{f}(1+k).
\end{align}
\begin{rema}
   Because $\mathcal{G}_{\eps,\beta}$ is self-adjoint, we have $B^{-1}_{-k,\beta} = B^{1}_{k,\beta}$ for all $k \in \mathbb{R}$, in agreement with  \eqref{Bneg1} and \eqref{Bpos1}. Moreover, we easily verify the following limits hold at infinite depth:  
   \begin{align*}
\lim_{h \rightarrow +\infty} B^{-1}_{k,\beta} &= \frac{1}{2} \left(\beta -(k-1) \sqrt{\beta +k^2}-\sqrt{\left(\beta +(k-1)^2\right) \left(\beta +k^2\right)}+k^2+k \sqrt{\beta +(k-1)^2}-k\right), \\
\lim_{h \rightarrow +\infty} B^{1}_{k,\beta} &= \frac{1}{2} \left(\beta +(k+1) \sqrt{\beta +k^2}-\sqrt{\left(\beta +k^2\right) \left(\beta +(k+1)^2\right)}+k^2-k \sqrt{\beta +(k+1)^2}+k\right),
   \end{align*}
   which are consistent with the coefficients obtained in \cite{CreNguStr}.
\end{rema}
The calculations leading to $R_{2,\beta}$ and $R_{3,\beta}$ follow the same procedure as those for $R_{0,\beta}$ and $R_{1,\beta}$ shown above. However, the expressions involved in these calculations become extremely lengthy and arduous to work with by hand. We instead appeal to Mathematica to help us finish these calculations. The interested reader will find the explicit expressions for the coefficients of $R_{2,\beta}$ and $R_{3,\beta}$ in our companion Mathematica file {\it CompanionToTransverseInstabilityFiniteDepth.m}.
\end{proof}
Next, we expand the operators  $\cH^j_\beta$ in \eqref{expandHepsbeta}  to third order in $\delta=\beta-\beta_*$.   By virtue of Theorem \ref{theo:flattenG}  (ii), for each $h>0$, the operators $R_{j,\beta}$ depend analytically on $\beta \in (0, \infty)$. In particular, we have
 \bq\label{def:Sjl}
 %\Omega=\sum_{\ell=0}^3\delta^\ell \Omega_\ell+O(\delta^4),\quad
  R_{j,\beta}=\sum_{\ell=0}^3\delta^\ell R^{j,\ell}+O(\delta^4),
 \eq
 where of course the $R^{j,\ell}$ depend  on $\beta_*$ and $h$. Noting that $\beta$ only appears in the lower right corner of $\cH^j_{\beta}$, we set
 \bq
 \begin{aligned}
&\cH^{j ,0}=\cH^j_{\beta}\vert_{\delta=0},\quad  \cH^{j, \ell}=R^{j, \ell}\begin{bmatrix} 0 & 0\\ 0 &1 \end{bmatrix},\quad \ell \in \{1, 2, 3\}, 
 \end{aligned}
 \eq
 so that 
  \bq\label{def:Hjl}
 \cH^{j}_{\beta}=\sum_{\ell=0}^3 \delta^\ell\cH^{j, \ell}+O(\delta^4). 
 \eq
Combining \eqref{Hexpansion:eps} and \eqref{def:Hjl}, we obtain the third-order  expansion of 
the Hamiltonian $\cH_{\eps, \beta_*+\delta}$   in both  $\eps$ and $\delta$:
 \bq\label{expand:H}
 \cH_{\eps, \beta_*+\delta}=\sum_{j=0}^3\sum_{\ell=0}^{3}\eps^j\delta^\ell \cH^{j, \ell}+O(\eps^4+\delta^4).
 \eq
% We are now poised to expand the spectral data of the $2\times 2$ matrix $\textrm{L}_{\varepsilon,\delta}$ as power series in both $\varepsilon$ and $\delta$.
%%%%%%%%%%%%%%%%%%%%%%%%%%%%%%%
\subsection{Third-order expansions  of  $\textrm{L}_{\eps, \delta}$}\label{section:expansions}
%\subsection{Expansion of the eigenvectors $U^{\eps, \delta}_j$} 
We recall the basis $\{U_1,U_2\}$ from \eqref{def:U} and the projections $P_{\eps,\delta}$ from \eqref{def:Peps}.  
Since $P_{0,0}U_j=U_j$ for $j=1,2$, we have 
\begin{align*}
U^{\eps, \delta}_j&= \{1-(P_{\eps, \delta}-P_{0,0})^2\}^{-\mez} P_{\eps, \delta} U_j.
\end{align*}
Next, we record the following third-order expansions of $U^{\eps, \delta}_j$ and $\textrm{L}_{\eps, \delta}$. We refer to   \cite{CreNguStr} for the derivation of these formulas. 
Denoting 
\bq
P^{m, n}=\p_\eps^m\p_\delta^nP_{\eps, \delta}\vert_{(\eps, \delta)=(0, 0)},
\eq
%we expand 
%\bq\label{expand:P}
%\begin{aligned}
%P_{\eps, \delta}&=P_{0,0}+\eps P^{1, 0}+\delta P^{0, 1}+\mez \eps^2 P^{2, 0}+\mez \delta^2 P^{0, 2}+\eps \delta P^{1, 1}\\
%&\qquad+\frac16 \eps^3 P^{3, 0}+\mez \eps^2\delta P^{2, 1}+\mez \eps \delta^2 P^{1, 2}+\frac16\delta^3P^{0, 3}+O((\eps+\delta)^4).
%\end{aligned}
%\eq
%This series and the ones that follow converge for small $(\eps,\delta)$ due to the discussion in Section \ref{Sec:Kato}.  
%Using the elementary Taylor expansion $(1-x^2)^{-\mez}=1+\mez x^2+O(x^4)$, we find 
%\bq\label{expand:Kato:1}
%\begin{aligned}
%\{I-(P_{\eps, \delta}-P_{0,0})^2\}^{-\mez}&=I +\mez \eps^2 P^{1, 0}P^{1, 0}+\mez \delta^2 P^{0, 1}P^{0, 1}+\mez \eps \delta(P^{1, 0}P^{0, 1}+P^{0, 1}P^{1, 0})\\
%&\quad+\frac14\eps^3(P^{1, 0}P^{2, 0}+P^{2, 0}P^{1, 0})+\frac14\delta^3(P^{0, 1}P^{0, 2}+P^{0, 2}P^{0, 1})\\
%&\quad+ \eps^2 \delta\left[\mez(P^{1, 0}P^{1, 1}+P^{1, 1}P^{1, 0})+\frac14 (P^{0, 1}P^{2, 0}+P^{2, 0}P^{0, 1})\right]\\
%&\quad+ \eps \delta^2\left[\mez(P^{0, 1}P^{1, 1}+P^{1, 1}P^{0, 1})+\frac14 (P^{1, 0}P^{0, 2}+P^{0, 2}P^{1, 0})\right]\\
%&\quad+O((\eps+\delta)^4).
%\end{aligned}
%\eq
%{\red ANALYTIC}
%%%%%%%%%%%%%%%%
%Combining the expansions \eqref{expand:P} and \eqref{expand:Kato:1}, we obtain the expansion of $U^{\eps, \delta}_j$ as 
we have 
\bq\label{expand:U}
U^{\eps, \delta}_j=U_j+\sum_{m+n=1}^3\eps^m\delta^nU^{(m, n)}_j+O((\eps+\delta)^4),
\eq
%It follows that 
%\bq\label{expand:U}
%\begin{aligned}
%U^{\eps, \delta}_j&=  \left\{1+\mez \eps^2 P^{1, 0}P^{1, 0}+\mez\delta^2 P^{0, 1}P^{0, 1}+\mez\eps \delta (P^{0, 1}P^{1, 0}+P^{1, 0}P^{0, 1})+O(\eps^3+\delta^3)\right\} \\
%& \quad\times \left\{P_{0,0}+\eps P^{1, 0}+\delta P^{0, 1}+\mez \eps^2 P^{2, 0}+\mez \delta^2 P^{0, 2}+\eps \delta P^{1, 1}+O(\eps^3+\delta^3)\right\}U_j\\
%&=U_j+\eps P^{1, 0}U_j+\delta P^{0, 1}U_j+\mez\eps^2(P^{2, 0}+P^{1, 0}P^{1, 0})U_j+\mez \delta^2(P^{0, 2}+P^{0, 1}P^{0, 1})U_j\\
%&\quad+\eps\delta (P^{1, 1}+\mez P^{0, 1}P^{1, 0}+\mez P^{1, 0}P^{0, 1})U_j+O(\eps^3+\delta^3)\\
%& :=U_j+\eps U^{(1, 0)}_j+\delta U^{(0, 1)}_j+\eps^2U^{(2, 0)}_j+\delta^2U^{(0, 2)}_j+\eps\delta U^{(1, 1)}_j+O(\eps^3+\delta^3)
%\end{aligned}
%\eq
where the coefficients $U_j^{(m, n)}$ are given by
 \allowdisplaybreaks
\begin{align}\label{expand:U:start}
&U^{(1, 0)}_j=P^{1, 0}U_j,\quad U^{(0, 1)}_j=P^{0, 1}U_j,\\
&U^{(2, 0)}_j=\mez (P^{2, 0}+P^{1, 0}P^{1, 0})U_j,\quad U^{(0, 2)}_j=\mez (P^{0, 2}+P^{0, 1}P^{0, 1})U_j,\\
&U^{(1, 1)}_j=(P^{1, 1}+\mez P^{0, 1}P^{1, 0}+\mez P^{1, 0}P^{0, 1})U_j,\\
&U_j^{(3,0)} = \frac16\Big(P^{3,0} + \frac32\left(P^{2,0}P^{1,0}+P^{1,0}P^{2,0} \right)+3P^{1,0}P^{1,0}P^{1,0} \Big)U_j, \\
&U_j^{(0,3)}= \frac16\Big(P^{0,3} + \frac32\left(P^{0,1}P^{0,2}+P^{0,2}P^{0,1} \right)+3P^{0,1}P^{0,1}P^{0,1} \Big)U_j,\\
&U_j^{(2,1)} = \frac12\Big(P^{2,1} + \frac12\left(P^{0,1}P^{2,0} +2P^{1,0}P^{1,1}+2P^{1,1}P^{1,0}+P^{2,0}P^{0,1}\right) \nonumber \\
&\hspace{1cm} +\left(P^{1,0}P^{0,1}+P^{0,1}P^{1,0} \right)P^{1,0} + P^{1,0}P^{1,0}P^{0,1}\Big)U_j,
\\ \label{expand:U:end}
&U_j^{(1,2)} = \frac12\Big(P^{1,2} + \frac12\left(P^{1,0}P^{0,2} +2P^{1,1}P^{0,1}+2P^{0,1}P^{1,1}+P^{0,2}P^{1,0}\right) \nonumber \\
&\hspace{1cm} +\left(P^{1,0}P^{0,1}+P^{0,1}P^{1,0} \right)P^{0,1} + P^{0,1}P^{0,1}P^{1,0}\Big)U_j.
\end{align}
Here $P^{m, n}U_j$ are computed from the following contour integrals:
 \allowdisplaybreaks
\begin{align}\label{Pmn:start}
P^{1, 0}U_j&= \frac1{2\pi i} \int_\Gamma -S_\ld J\cH^{1, 0}U_j \frac{d\lb}{\ld-i\sigma},\\
 P^{0, 1}U_j&= \frac1{2\pi i} \int_\Gamma -S_\ld J\cH^{0, 1}U_j \frac{d\lb}{\ld-i\sigma},\\
P^{2, 0}U_j&=\frac1{2\pi i} \int_\Gamma S_\ld \Big(-2J\cH^{2, 0}+2J\cH^{1, 0} J\cH^{1, 0}\Big)U_j\frac{d\lb}{\ld-i\sigma},\\
P^{0, 2}U_j&=\frac1{2\pi i} \int_\Gamma S_\ld \Big(-2J\cH^{0, 2}+2 J\cH^{0, 1}S_\ld J\cH^{0, 1}\Big)U_j\frac{d\lb}{\ld-i\sigma},\\
P^{1, 1}U_j&=\frac1{2\pi i} \int_\Gamma S_\ld\Big(- J\cH^{1, 1}+ J\cH^{1, 0}S_\ld J\cH^{0, 1}+ J\cH^{0, 1}S_\ld J\cH^{1, 0}\Big)U_j\frac{d\lb}{\ld-i\sigma},\\
P^{3,0}U_j &= \frac{1}{2\pi i}\int_\Gamma 6S_\lambda\Big[-J H^{3,0} + \big(JH^{1,0}S_\lambda JH^{2,0} + JH^{2,0}S_\lambda JH^{1,0}\big) \nonumber\\
&\hspace{1cm}- JH^{1,0}S_\lambda JH^{1,0}S_\lambda JH^{1,0}\Big]U_j \frac{d\lb}{\ld-i\sigma}, 
\\
P^{0,3}U_j &= \frac{1}{2\pi i}\int_\Gamma 6S_\lambda\Big[-J H^{0,3} + \big(JH^{0,1}S_\lambda JH^{0,2} + JH^{0, 2}S_\lambda JH^{0, 1}\big) \nonumber\\
&\hspace{1cm}- JH^{0,1}S_\lambda JH^{0, 1}S_\lambda JH^{0, 1}\Big]U_j \frac{d\lb}{\ld-i\sigma},\\
P^{2,1}U_j &= \frac{1}{2\pi i}\int_\Gamma 2S_\lambda\Big[-J H^{2,1} +JH^{1,0}S_\lambda JH^{1,1} + JH^{0,1}S_\lambda JH^{2,0} \nonumber \\
&\hspace{1cm} + JH^{2,0}S_\lambda JH^{0,1} + JH^{1,1}S_\lambda JH^{1,0}  \nonumber
\\
&\hspace{1cm}  -JH^{1,0}S_\lambda \big(JH^{1,0}S_\lambda JH^{0,1}+ JH^{0,1}S_\lambda JH^{1,0}\big)\Big]U_j \frac{d\lb}{\ld-i\sigma},\\ \label{Pmn:end}
P^{1,2}U_j &= \frac{1}{2\pi i}\int_\Gamma 2S_\lambda\Big[-J H^{1, 2} +JH^{0, 1}S_\lambda JH^{1,1} + JH^{1, 0}S_\lambda JH^{0, 2} \nonumber \\
&\hspace{1cm} + JH^{0, 2}S_\lambda JH^{1, 0} + JH^{1,1}S_\lambda JH^{0, 1}  \nonumber
\\
&\hspace{1cm}  -JH^{0, 1}S_\lambda \big(JH^{0, 1}S_\lambda JH^{1, 0}+ JH^{0,1}S_\lambda JH^{0, 1}\big)\Big]U_j \frac{d\lb}{\ld-i\sigma}.
\end{align}
%To summarize, the third-order expansions of the eigenvectors $U_j^{\eps, \delta}$ are given by \eqref{expand:U}, where the coefficients $U_j^{(m, n)}$ are expressed in terms of $P^{m, n}$ as in \eqref{expand:U:start}-\eqref{expand:U:end} which are in turn calculated by the contour integrals in \eqref{Pmn:start}-\eqref{Pmn:end}.
%%%%%%%%%%%%%%%%%%%%
%\subsection{Expansions of the matrix $\textrm{L}_{\eps, \delta}$}
We recall that   $\textrm{L}_{\eps, \delta}$ given by \eqref{matrixL}  is the matrix representation of the linearized operator $\mathcal{L}_{\eps, \beta_*+\delta}$ with respect to the basis $\{V_j^{\eps, \delta}: j=1, 2\}$, where  $V_j^{\eps, \delta}=\frac{1}{\sqrt{\gamma_j}}U_j^{\eps, \delta}$. It follows from the expansion  \eqref{expand:U} of $U_j^{\eps, \delta}$ that
\bq\label{expand:V}
V^{\eps, \delta}_j=V_j+\sum_{m+n=1}^3\eps^m\delta^nV_j^{(m, n)}+O((\eps+\delta)^4),\quad V_j=\frac{1}{\sqrt{\gamma_j}}U_j,\quad V_j^{(m, n)}=\frac{1}{\sqrt{\gamma_j}}U_j^{(m, n)}.
\eq
Combining \eqref{expand:V} with the expansion \eqref{expand:H} for $\cH_{\eps, \beta_*+ \delta}$, we  expand  the inner products appearing in $L_{\eps, \delta}$ \eqref{matrixL} as
\bq
\left(\mathcal{H}_{\varepsilon,\beta_*+\delta}V_j^{\varepsilon,\delta},V_k^{\varepsilon,\delta} \right)=(\cH^{0, 0}V_j, V_k)+\sum_{m+n=1}^3 \left(\mathcal{H}_{\varepsilon,\beta_*+\delta}V_j^{\varepsilon,\delta},V_k^{\varepsilon,\delta} \right)_{m,n}\eps^m\delta^n+O((\eps+\delta)^4),
\eq
where
 \allowdisplaybreaks
\begin{align}\label{innerproduct:start}
 \left(\mathcal{H}_{\varepsilon,\beta_*+\delta}V_j^{\varepsilon,\delta},V_k^{\varepsilon,\delta} \right)_{1, 0}&=(\cH^{0, 0}V_j, V^{(1, 0)}_k)+(\cH^{1, 0}V_j, V_k)+(\cH^{0, 0}V_j^{(1, 0)}, V_k),\\
 \left(\mathcal{H}_{\varepsilon,\beta_*+\delta}V_i^{\varepsilon,\delta},V_{j}^{\varepsilon,\delta} \right)_{0, 1}&=(\cH^{0, 0}V_j, V^{(0, 1)}_k)+(\cH^{0, 1}V_j, V_k)+(\cH^{0, 0}V_j^{(0, 1)}, V_k),\\
 \left(\mathcal{H}_{\varepsilon,\beta_*+\delta}V_i^{\varepsilon,\delta},V_{j}^{\varepsilon,\delta} \right)_{2, 0}&=(\cH^{0, 0}V_j, V^{(2, 0)}_k)+(\cH^{2, 0}V_j, V_k) +(\cH^{1, 0}V^{(1, 0)}_j, V_k)\nonumber\\
&\hspace{1cm} +(\cH^{0, 0}V^{(2, 0)}_j, V_k)+(\cH^{1, 0}V_j, V^{(1, 0)}_k)+(\cH^{0, 0}V_j^{(1, 0)}, V_k^{(1, 0)}),\\
 \left(\mathcal{H}_{\varepsilon,\beta_*+\delta}V_i^{\varepsilon,\delta},V_{j}^{\varepsilon,\delta} \right)_{0, 2}&=(\cH^{0, 0}V_j, V^{(0, 2)}_k)+(\cH^{0, 2}V_j, V_k) +(\cH^{0, 1}V^{(0, 1)}_j, V_k)\nonumber\\
&\hspace{1cm}+(\cH^{0, 0}V^{(0, 2)}_j, V_k)+(\cH^{0, 1}V_j, V^{(0, 1)}_k)+(\cH^{0, 0}V_j^{(0, 1)}, V_k^{(0, 1)}),\\
 \left(\mathcal{H}_{\varepsilon,\beta_*+\delta}V_i^{\varepsilon,\delta},V_{j}^{\varepsilon,\delta} \right)_{1, 1}&=(\cH^{1, 1}V_j, V_k)+(\cH^{1, 0}V_j^{(0, 1)}, V_k)+(\cH^{0, 1}V_j^{(1, 0)}, V_k)\nonumber\\
&\hspace{1cm} +(\cH^{0, 0}V_j, V_k^{(1, 1)}) +(\cH^{0, 0}V_j^{(1, 1)}, V_k) +(\cH^{1, 0}V_j, V_k^{(0, 1)})\nonumber\\
&\quad+(\cH^{0, 0}V_j^{(1, 0)}, V_k^{(0, 1)})+(\cH^{0, 1}V_j, V_k^{(1, 0)})+(\cH^{0, 0}V^{(0, 1)}_j, V_k^{(1, 0)}),\\
\left(\mathcal{H}_{\varepsilon,\beta_*+\delta} V^{\varepsilon,\delta}_j,V^{\varepsilon,\delta}_k \right)_{3,0} &=\Big(\mathcal{H}^{3,0} V_j + \mathcal{H}^{2,0}V_j^{(1,0)} + \mathcal{H}^{1,0}V_j^{(2,0)} + \mathcal{H}^{0,0}V_j^{(3,0)},V_k\Big) \nonumber \\
&\hspace{1cm}+\Big(\mathcal{H}^{2,0}V_j + \mathcal{H}^{1,0}V_j^{(1,0)} + \mathcal{H}^{0,0}V_j^{(2,0)},V_k^{(1,0)}   \Big) \nonumber \\
&\hspace{1cm}+ \Big(\mathcal{H}^{1,0}V_j + \mathcal{H}^{0,0}V_j^{(1,0)},V_k^{(2,0)}   \Big) + \Big(\mathcal{H}^{0,0}V_j,V_k^{(3,0)}  \Big), \\
\left(\mathcal{H}_{\varepsilon,\beta_*+\delta} V^{\varepsilon,\delta}_j,V^{\varepsilon,\delta}_k \right)_{0,3} &=\Big(\mathcal{H}^{0,3} V_j + \mathcal{H}^{0,2}V_j^{(0,1)} + \mathcal{H}^{0,1}V_j^{(0,2)} + \mathcal{H}^{0,0}V_j^{(0,3)},V_k\Big) \nonumber \\
&\hspace{1cm}+\Big(\mathcal{H}^{0,2}V_j + \mathcal{H}^{0,1}V_j^{(0,1)} + \mathcal{H}^{0,0}V_j^{(0,2)},V_k^{(0,1)}   \Big) \nonumber \\
&\hspace{1cm}+ \Big(\mathcal{H}^{0,1}V_j + \mathcal{H}^{0,0}V_j^{(0,1)},V_k^{(0,2)}   \Big) + \Big(\mathcal{H}^{0,0}V_j,V_k^{(0,3)}  \Big),\\
\left(\mathcal{H}_{\varepsilon,\beta_*+\delta} V^{\varepsilon,\delta}_j,V^{\varepsilon,\delta}_k \right)_{2,1} &= \Big(\mathcal{H}^{2,1}V_j + \mathcal{H}^{2,0}V_j^{(0,1)} + \mathcal{H}^{1,1}V_j^{(1,0)} + \mathcal{H}^{1,0}V_j^{(1,1)} + \mathcal{H}^{0,1}V_j^{(2,0)} + \mathcal{H}^{0,0}V_j^{(2,1)}, V_k \Big) \nonumber \\
&\hspace{1cm}+\Big(\mathcal{H}^{2,0}V_j + \mathcal{H}^{1,0}V_j^{(1,0)} + \mathcal{H}^{0,0}V_j^{(2,0)},V_k^{(0,1)} \Big) \nonumber \\
&\hspace{1cm}+\Big(\mathcal{H}^{1,1}V_j + \mathcal{H}^{1,0}V_j^{(0,1)} + \mathcal{H}^{0,1}V_j^{(1,0)} + \mathcal{H}^{0,0}V_j^{(1,1)},V_{n}^{(1,0)} \Big) \nonumber \\
&\hspace{1cm}+ \Big(\mathcal{H}^{1,0}V_j + \mathcal{H}^{0,0}V_j^{(1,0)},V_k^{(1,1)}  \Big) + \Big(\mathcal{H}^{0,1}V_j + \mathcal{H}^{0,0}V_j^{(0,1)},V_k^{(2,0)} \Big) \nonumber \\
&\hspace{1cm}+ \Big(\mathcal{H}^{0,0}V_j,V_k^{(2,1)} \Big), \\ \label{innerproduct:end}
\left(\mathcal{H}_{\varepsilon,\beta_*+\delta} V^{\varepsilon,\delta}_j,V^{\varepsilon,\delta}_k \right)_{1,2} &= \Big(\mathcal{H}^{1,2}V_j + \mathcal{H}^{1,1}V_j^{(0,1)} + \mathcal{H}^{1,0}V_j^{(0,2)} + \mathcal{H}^{0,2}V_j^{(1,0)} + \mathcal{H}^{0,1}V_j^{(1,1)} + \mathcal{H}^{0,0}V_j^{(1,2)}, V_k \Big) \nonumber \\
&\hspace{1cm}+\Big(\mathcal{H}^{0,2}V_j + \mathcal{H}^{0,1}V_j^{(0,1)} + \mathcal{H}^{0,0}V_j^{(0,2)},V_k^{(1,0)} \Big) \nonumber \\
&\hspace{1cm}+\Big(\mathcal{H}^{1,1}V_j + \mathcal{H}^{1,0}V_j^{(0,1)} + \mathcal{H}^{0,1}V_j^{(1,0)} + \mathcal{H}^{0,0}V_j^{(1,1)},V_{n}^{(0,1)} \Big) \nonumber \\
&\hspace{1cm}+ \Big(\mathcal{H}^{0,1}V_j + \mathcal{H}^{0,0}V_j^{(0,1)},V_k^{(1,1)}  \Big) + \Big(\mathcal{H}^{1,0}V_j + \mathcal{H}^{0,0}V_j^{(1,0)},V_k^{(0,2)} \Big) \nonumber \\
&\hspace{1cm}+ \Big(\mathcal{H}^{0,0}V_j,V_k^{(1,2)} \Big).
\end{align}

%%%%%%%%%%%%%%%%%%%%%%%%%%% 
%%%%%%%%%%%%%%%%%%%%%%%%%%% 

%%%%Previous proof%%%%%%%%%%%%
% Fixing  $\eps$ sufficiently small, we can approximate the location of the eigenvalues $\ld_\pm(\eps, \delta)$ as $\delta$ varies.  To this end,  we  first recall from \eqref{def:Delta3} that
%\[
% \delta = \left(\frac{-a_{2,0}+c_{2,0}}{a_{0,1}-c_{0,1}}\right)\varepsilon^2 + \theta\varepsilon^3 + O\left(\varepsilon^4 \right),\label{delta_formal_expansion}
% \]
%where $\theta$ can be any real number such that
%\begin{align}
% \left|\theta \right|  <  \left|\frac{2b_{3,0}}{c_{0,1}-a_{0,1}}\right|.
%\end{align}
%Substituting, 
%%expansions \eqref{MNQ_expansions}, \eqref{T_expansion}, and \eqref{delta_formal_expansion} into \eqref{lambda_exact}, 
%we obtain the asymptotic approximation 
%\begin{align} \allowdisplaybreaks
%\lambda_\pm(\varepsilon) &= i\left[\sigma + \left(\frac{a_{0,1}c_{2,0}-a_{2,0}c_{0,1}}{a_{0,1}-c_{0,1}}\right)\varepsilon^2 + \theta\left(\frac{a_{0,1}+c_{0,1}}{2} \right) \varepsilon^3 \right] \nonumber \\ &\hspace{1cm} \pm \frac{\varepsilon^3}{2} \sqrt{4b_{3,0}^2 - \theta^2(a_{0,1}-c_{0,1})^2
%  }+O\left(\varepsilon^4\right). \label{lambda_approx}
%\end{align}
%Here the  eigenvalues $\ld_\pm(\eps)$ are parameterized by $\theta$.   
%By eliminating $\theta$ from the real and imaginary parts of \eqref{lambda_approx}, we  find 
%\eqref{form:ellipse}. 
\section{Proof of Theorem \ref{theo:main}}\label{Sec:proof}

\begin{lemm} \label{freqProp}
The purely imaginary matrix $\textrm{L}_{\eps, \delta}$ can be written as  
%Substituting the expansions \eqref{innerproduct:start}-\eqref{innerproduct:end} of the inner-products $(\mathcal{H}_{\varepsilon,\beta_*+\delta}V_j^{\varepsilon,\delta},V_k^{\varepsilon,\delta})$ into the matrix $L_{\eps, \delta}$ \eqref{matrixL}, we find
\begin{align}
\textrm{L}_{\varepsilon,\delta} = 
%\begin{pmatrix} i\sigma & 0 \\ 0 & i\sigma \end{pmatrix} + \textrm{M}_{\varepsilon,\delta}     := 
\begin{pmatrix} i\sigma & 0 \\ 0 & i\sigma \end{pmatrix} + i\begin{pmatrix} A &B \\ -B& C \end{pmatrix}, 
\end{align}
where $i\sigma$ is the repeated eigenvalue of the unperturbed operator $\mathcal{L}_{0,\beta_*}$ and
%\begin{align}
%\textrm{M}_{\varepsilon,\delta} &= i\begin{pmatrix} A &B \\ -B& C \end{pmatrix}, %\label{Meqn}
%\end{align}
where $A$, $B$, and $C$ are real analytic functions of $(\eps, \delta)$ in a neighborhood of $(0,0)$ that have the expansions 
\bq\label{expand:ABC}
\begin{aligned}
A &= a_{0,1}\delta + a_{2,0}\varepsilon^2 + a_{0,2}\delta^2 + a_{2,1}\varepsilon^2\delta +a_{0,3}\delta^3 +a_{4,0}\eps^4   + O(\delta(|\eps|^3+|\delta|^3)),\\
B&=b_{3,0}\varepsilon^3 + O(|\eps|^4+|\delta|^4),\\
C &=c_{0,1}\delta + c_{2,0}\varepsilon^2 + c_{0,2}\delta^2 + c_{2,1}\varepsilon^2\delta +c_{0,3}\delta^3 +c_{4,0}\eps^4 +  O(\delta (|\eps|^3+|\delta|^3)).
\end{aligned}
\eq
We emphasize that all these coefficients depend analytically on the depth $h\in (0, \infty)$. 
\end{lemm} 
%%%%%%%%%%%%%%%%%%%%%%%
\begin{proof} 
The proof is identical to that in the infinite depth case \cite{CreNguStr}. In particular, since $\textrm{L}_{\varepsilon,\delta}$ is purely imaginary and real analytic in $\varepsilon$ and $\delta$, we have that $A$, $B$, and $C$ are real-valued, real-analytic functions of $\varepsilon$ and $\delta$. The anti-symmetry of the off-diagonal entries of $\textrm{L}_{\varepsilon,\delta}$ follows from  \eqref{L:reversediagonal}. Expansions \eqref{expand:ABC} come from 
substituting the expansions \eqref{innerproduct:start}-\eqref{innerproduct:end} of the inner products $(\mathcal{H}_{\varepsilon,\beta_*+\delta}V_j^{\varepsilon,\delta},V_k^{\varepsilon,\delta})$ into the matrix $L_{\eps, \delta}$ \eqref{matrixL}. As in the infinite depth case, each term in the expansion of $V_j^{\varepsilon,\delta}$ consists of a finite Fourier series spanning a small set of wave numbers given in the table below. Terms in the expansion of $\mathcal{H}_{\varepsilon,\beta_*+\delta}$ act to modulate these wave numbers in such a way that many terms in the expansion of $(\mathcal{H}_{\varepsilon,\beta_*+\delta}V_j^{\varepsilon,\delta},V_k^{\varepsilon,\delta})$ vanish. Those that do not vanish yield the coefficients $a_{i,j}$, $b_{i,j}$, and $c_{i,j}$ that appear above. These coefficients are real analytic functions of the resonant value $\beta_*$ and the depth parameter $h$. Their explicit expressions, except for $a_{4,0}$ and $c_{4, 0}$, can be found in our Mathematica file {\it CompanionToTransverseInstabilityFiniteDepth.nb}. It will turn out that the fourth-order coefficients {\it $a_{4,0}$ and $c_{4, 0}$ are not needed} for the proof of Theorem \ref{theo:main}.
\begin{center}
\begin{tabular}{ p{2.5cm}|p{2.5cm}||p{2.5cm}|p{2.5cm}  }
 \multicolumn{2}{c}{$V_1^{(k,\ell)}$ Wave Numbers} & \multicolumn{2}{c}{$V_2^{(k,\ell)}$ Wave Numbers}  \\ 
 \hline \hline
 $V_1$ & $\{1\}$ & $V_2$ & $\{-2\}$\\
 \hline
  $V_1^{(1,0)}$ & $\{0,2\}$ & $V_2^{(1,0)}$ & $\{-3,-1\}$\\
 \hline
   $V_1^{(0,1)}$ & $\{1\}$ & $V_2^{(0,1)}$ & $\{-2\}$\\
   \hline
     $V_1^{(2,0)}$ & $\{-1,1,3\}$ & $V_2^{(2,0)}$ & $\{-4,-2,0\}$\\
 \hline
   $V_1^{(1,1)}$ & $\{0,2\}$ & $V_2^{(1,1)}$ & $\{-3,-1\}$\\
 \hline
   $V_1^{(0,2)}$ & $\{1\}$ & $V_2^{(0,2)}$ & $\{-2\}$\\
   \hline
     $V_1^{(3,0)}$ & $\{-2,0,2,4\}$ & $V_2^{(3,0)}$ & $\{-5,-3,-1,1\}$\\
 \hline
   $V_1^{(2,1)}$ & $\{-1,1,3\}$ & $V_2^{(2,1)}$ & $\{-4,-2,0\}$\\
 \hline
   $V_1^{(1,2)}$ & $\{0,2\}$ & $V_2^{(1,2)}$ & $\{-3,-1\}$\\
 \hline
   $V_1^{(0,3)}$ & $\{1\}$ & $V_2^{(0,3)}$ & $\{-2\}$\\
 \hline
 \hline
\end{tabular}
\end{center}
%listed in the proposition above, completing our proof. 
\end{proof} 
As it turns out, most of the coefficients $a_{i,j}$, $b_{i,j}$, and $c_{i,j}$ in expansions \eqref{expand:ABC} do not play a central role in the stability calculations that follow. However, we will need both that $a_{0, 1}\ne c_{0, 1}$  and that $b_{3,0}$ is not identically zero for $h > 0$, which we prove in the following lemmata. 
\begin{lemm}\label{lemm:b30}
For all $h > 0$, $a_{0,1}<0< c_{0,1}$. 
\end{lemm}
\begin{proof}
    From {\it CompanionToTransverseInstabilityFiniteDepth.nb}, we have
    \begin{align*}
a_{0,1} = -\frac{\tau_1}{2\gamma_1} \quad \textrm{and} \quad c_{0,1} = \frac{\tau_2}{2\gamma_2},
    \end{align*} 
where $\gamma_j>0$ is given in \eqref{def:gammaj} and 
    \begin{align*}
\tau_j :=\frac12 \left(h\ \textrm{sech}^2\left(h\left( j^2+\beta_*\right)^\mez \right) + \frac{\tanh(h\left(j^2+\beta_*\right)^\mez)}{\left(j^2+\beta_*\right)^\mez}\right)>0.
    \end{align*}
Thus  $a_{0,1}<0< c_{0,1}$.
\end{proof}

\begin{lemm}\label{lemm:b30v2}
The coefficient $b_{3,0}$ exhibits the following limiting behaviors:
\begin{align*} b_{3,0}(h) = \frac{9}{1024\sqrt{2}} h^{-9/2} + \mathscr{o}\left(h^{-9/2}\right) ~ \textrm{as} ~h \rightarrow 0^+ \quad \textrm{and} \quad \displaystyle \lim_{h \rightarrow \infty} b_{3,0}(h) = b_{3,0,\infty} < 0. \end{align*}
Consequently, $b_{3,0}$ has only a finite number of zeros over $h > 0$. In fact, $b_{3,0}$ has at least one zero at $h_{\textrm{crit}} = 0.25065...$.
\end{lemm}
\begin{proof}
 We require computer assistance to prove the limiting behavior of $b_{3,0}$ as $h \rightarrow 0^+$ and $h \rightarrow \infty$. However, at no point in our computer-assisted proof do we rely on numerical approximations. We work directly with the exact expression for $b_{3,0}$ given in {\it CompanionToTransverseInstabilityFiniteDepth.nb}.

  For the limit $h \rightarrow 0^+$, we substitute the expansion $\beta_*(h) \sim 4/3 h^2$ from Proposition ~\ref{prop:rescond} into the exact expression of $b_{3,0}$. What remains is a function only of $h$. With the help of Mathematica's Series function, we then perform a Puiseux series expansion of this function of $h$ about $h = 0$ and, in particular, obtain the leading-order term of this expansion, giving us
  \begin{align*}
 b_{3,0}(h) = \frac{9}{1024\sqrt{2}} h^{-9/2} + \mathscr{o}\left(h^{-9/2}\right) ~ \textrm{as} ~h \rightarrow 0^+ .
  \end{align*}
 The interested reader can verify this result in {\it CompanionToTransverseInstabilityFiniteDepth.nb}.
  
  For the limit $h \rightarrow \infty$, we refer to $\beta_*(h) \rightarrow \beta_{*,\infty}$ from Proposition ~\ref{prop:rescond} in combination with Mathematica's Limit function to calculate $b_{3,0,\infty}$ explicitly. The resulting expression we obtain for $b_{3,0,\infty}$ is initially quite complicated. However, we can compare this expression with the $b_{3,0}$ obtained in our infinite depth paper \cite{CreNguStr}. In particular, we show using Mathematica that the expression we obtain for $b_{3,0,\infty}$ in fact equals the $b_{3,0}$ of the infinite depth case. Using the expression for $b_{3,0}$ from our work in infinite depth in place of what we initially calculate for $b_{3,0,\infty}$ in this work, we find
\begin{align}
b_{3,0,\infty} = -\frac{(1+\xi^2)\left(p(\xi) + q(\xi)\sqrt{\xi^4 - 1} \right)}{r(\xi)},\label{b30_expression}
\end{align}
where
\begin{subequations}
\begin{align}
\xi &=  \frac12\left(3 - 3\left(2\sqrt{7}-1 \right)^{-\frac13} + \left( 2\sqrt{7}-1\right)^{\frac13} \right),  \label{xi_inf_depth}  \\
p(\xi) &= 87016 - 379693 \xi + 739050 \xi^2 - 974082 \xi^3 + 1042548 \xi^4 -  898707 \xi^5 + 559890 \xi^6 - 214548 \xi^7 \nonumber \\ &\hspace{1.5cm} + 22208 \xi^8 + 25957 \xi^9 - 15538 \xi^{10} + 3670 \xi^{11} + 100 \xi^{12} - 229 \xi^{13} + 54 \xi^{14}, \label{p_xi_inf_depth}\\
q(\xi) &= -3760 - 127707 \xi + 448476 \xi^2 - 659689 \xi^3 + 526682 \xi^4 -  230882 \xi^5 + 32188 \xi^6 + 23322 \xi^7 \nonumber \\ &\hspace{1.5cm} - 14544 \xi^8 + 3149 \xi^9 + 
 280 \xi^{10} - 257 \xi^{11} + 54 \xi^{12}, \\
r(\xi) &= 64\left(\xi^2 -3\xi + 5\right)\left(\xi^2 -3\xi+6\right)\left(\xi^2 -6\xi + 11 + \sqrt{\xi^4-1}\right)\Big(\xi^3 - \xi^2 + \xi -1 \nonumber \\ &\hspace{1.5cm}+ \xi\sqrt{\xi^4 - 1}\Big)^2\sqrt{\xi(3-\xi)}.  
\end{align}
\end{subequations}
Note that the variable $\xi$ in this paper is equivalent to the variable $\gamma_1$ in our infinite depth paper.

Directly substituting \eqref{xi_inf_depth} into \eqref{p_xi_inf_depth}, we find
 \begin{align*}
p(\xi) = \frac{-97739871069 + p_1 + p_2+ p_3}{4096\left(2\sqrt{7}-1\right)^{\frac{14}{3}}},
 \end{align*}
 where
 \begin{align*}
 p_1 &= 78384770484\sqrt{7}, \\
p_2 &= -37428273\left(2\sqrt{7}-1\right)^\frac23\left( 953-232\sqrt{7} \right), \quad \textrm{and} \\
p_3 &= -3\left(2\sqrt{7}-1 \right)^{\frac13}\left(64260927809-20320120798\sqrt{7} \right).
 \end{align*}
 Now we have the simple bounds
 \begin{align*} 2.6 < \sqrt{7}, \quad \left(2\sqrt{7}-1\right)^{\frac13} < 1.6, \quad \textrm{and} \quad \left(2\sqrt{7} -1\right)^{\frac23} < 2.7 ,  
 \end{align*}
so that we obtain the following  {\it rational} lower bounds for $p_1$, $p_2$, and $p_3$:
 \begin{align*}
 (78384770484)\cdot(2.6) < p_1, \\
 -(37428273)\cdot(2.7)\cdot(953-232\cdot2.6) < p_2, \\
 -(3)\cdot(1.6)\cdot(64260927809-20320120798\cdot2.6) < p_3,
 \end{align*}
from which it follows that
 \begin{align*}
p(\xi) > 0.
 \end{align*}
A similar argument shows that $q(\xi) > 0$, and in \cite{CreNguStr} we show that $r(\xi) > 0$ and $1 < \xi < 2$. It follows that
\begin{align*}
b_{3,0,\infty} < 0. 
\end{align*}
In fact, if we numerically evaluate $b_{3,0,\infty}$, we find $b_{3,0,\infty} = -0.49476...$.

Given the behavior of $b_{3,0}$ as $h \rightarrow 0^+$ and $h \rightarrow \infty$, $b_{3, 0}$ must have at least one zero in $(0, \infty)$. %the existence of at least one value  $h_{crit} > 0$ for which $b_{3,0}(h_{crit}) = 0$. 
Since $b_{3,0}$ is analytic on $(0, \infty)$, its zero set $\mathcal{Z}$ cannot have an accumulation point in $(0, \infty)$.    $\mathcal{Z}$ is bounded because  $\lim_{h\to \infty} b_{3,0}(h)<0$. Moreover, $0$ cannot be an accumulation point of $\mathcal{Z}$ because  $\lim_{h\to 0^+} b_{3,0}(h)=\infty$. Thus $\mathcal{Z}$ is finite.

Although by analytic reasoning we have not been able to determine the number of zeros of $b_{3,0}$, 
a numerical calculation indicates that there is exactly one.  
Indeed, the numerical plot of $b_{3,0}$ in Figure 3 reveals that exactly one such depth occurs at $h_{crit} = 0.25065...$.  
In order to obtain Figure 3, we use Newton's method to compute $\beta_*(h)$ for a given $h$ from the resonance condition \eqref{resonancecond}. This in turn allows us to evaluate $b_{3,0}$ numerically at depth $h$. Repeating this procedure for $0.1 < h < 4$ in step sizes of $\Delta h = 0.01$ yields Figure 3. The critical depth $h_{crit} = 0.25065...$ is obtained by bisection using increasing smaller $\Delta h$ near the critical depth.

%We argue by contradiction that $b_{3,0}$ has at most a finite number of isolated zeros over $h > 0$. If instead we have an infinite number of zeros $h_j$, where $j$ is either a countable or uncountable index, then either $+\infty$ is an accumulation point of these zeros, or we have an infinite number of bounded zeros. The former case is impossible, since $\lim_{h \rightarrow \infty} b_{3,0}(h) = b_{3,0,\infty} < 0$. In the latter case, we could construct a convergent subsequence of these zeros $h_{j_k}$. Define the limit of this subsequence to be $h_*$. Then, the real analytic functions $b_{3,0}$ and $0$ are equal to each other over $M = \{h_{j_k} \} \subset \mathbb{R}^+$. However, $M$ has the accumulation point $h_*$, and by the Identity Theorem for real analytic functions (Proposition 1.1.19, \cite{Lewis}), we must have $b_{3,0} \equiv 0$ for all $h > 0$. This contradicts the limiting behavior found above, so $b_{3,0}$ has a finite number of isolated zeros. 

\begin{figure}[tb]
\includegraphics[width=9cm]{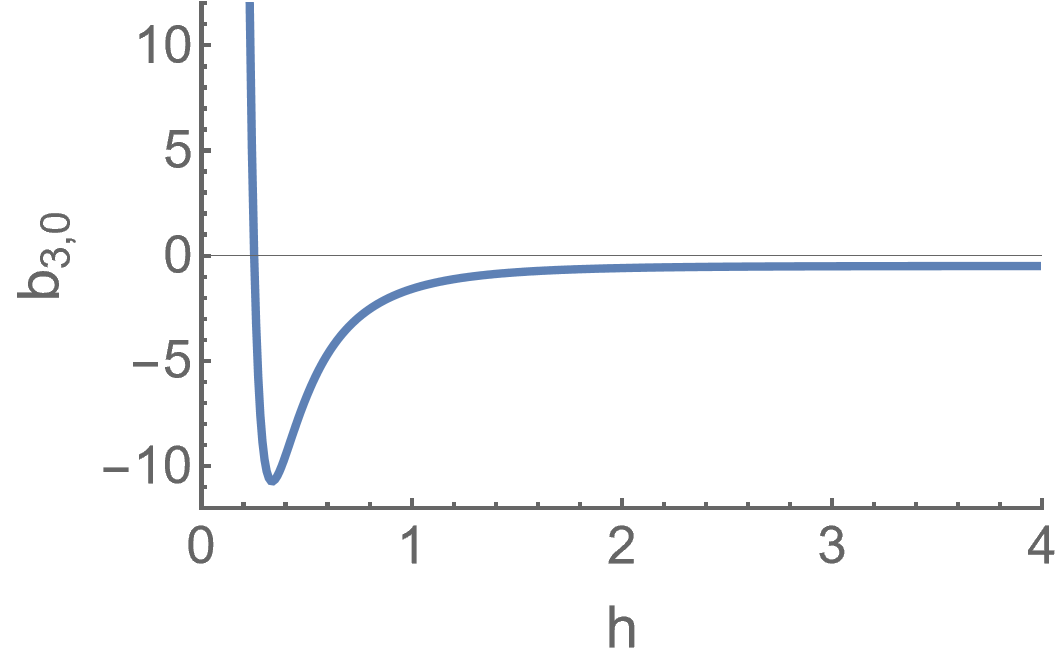}
\label{fig3}
\caption{A plot of $b_{3,0}$ as a function of $h$. As $h \rightarrow 0^+$, we have $b_{3,0} \rightarrow +\infty$. As $h \rightarrow \infty$, we have $b_{3,0} \rightarrow -0.49476...$, in agreement with \cite{CreNguStr}. At $ h_{crit} = 0.25065....$, we have $b_{3,0} = 0$. Thus, for this depth, we observe no transverse instability at $O(\eps^3)$. }
\end{figure}
\end{proof}

%%%%%%%%%%%%%%%%%%%%%%%
%\section*{\large Unstable Eigenvalues of $\textrm{L}_{\varepsilon,\beta_*+\delta}$ }
%%%%%%%%%%%%%%%%%%%%%%%
%\section*{\large Unstable Eigenvalues of $\textrm{L}_{\varepsilon,\beta_*+\delta}$ }

\large{ \textbf{ Proof of Theorem \ref{theo:main}}}
\vspace{.5cm}\\
The characteristic polynomial of $\textrm{L}_{\eps, \delta} - i\sigma I$ is 
\bq\label{charL} 
\det (\textrm{L}_{\eps, \delta}- i\sigma I -\lambda I) = \ld^2-i(A+C)\ld -AC-B^2 ,
\eq
whose discriminant is the real-valued function
  \bq\label{def:Deltaed}
  \Delta(\varepsilon,\delta) =-(A-C)^2+4B^2.
  \eq
  The eigenvalues of $\textrm{L}_{\eps, \delta}$ are 
\bq\label{ld:ABCDelta}
\lambda_\pm = i\Big(\sigma +\frac{1}2 (A+C) \Big) \pm \frac12\sqrt{\Delta(\eps,\delta)},  
\eq 
where we recall that $\sigma +\frac{1}2 (A+C)$ is real. Since $A$, $B$, and $C$ are real analytic in $(\eps, \delta)$ in a neighborhood of $(0, 0)$, so are $A+C$ and 
$\Delta(\eps, \delta)$. Moreover, it follows from the expansions of $A$, $B$, and $C$ that $\Delta(\varepsilon,\delta) = \cO(\delta^2)$ and $A+C = \cO(\delta)$ as $(\varepsilon,\delta) \rightarrow (0,0)$. We will complete the proof of Theorem \ref{theo:main} by {\it showing that the characteristic polynomial \eqref{charL} has a root with positive real part}, that is, $\Delta(\eps,\delta)>0$ for suitably small $\eps$ and $\delta$.  

Taking inspiration from our work in infinite depth \cite{CreNguStr}, we map $\delta$ to a new variable $\theta$ as follows:
\bq\label{delta:eps}
\delta = \kappa_0\eps^2 + \theta \eps^3, \quad \textrm{where} \quad \kappa_0 := \frac{a_{2,0}-c_{2,0}}{a_{0,1}-c_{0,1}}.
\eq
This transformation is well-defined by Lemma \ref{lemm:b30}.
Upon substituting \eqref{delta:eps} into the series expansions of $A$, $B$, and $C$ \eqref{expand:ABC}, we obtain the following expansion for $\Delta$:
\begin{align}
\Delta(\varepsilon,\delta) = \Delta(\varepsilon,\kappa_0\eps^2 + \theta \eps^3) = \varepsilon^6 \Big[4b_{3,0}^2 - (a_{0,1}-c_{0,1})^2\theta^2 \Big] + O\left(\varepsilon^7 \right). \label{deltaExp}
\end{align}
We refer to Section 6 in  \cite{CreNguStr} for the detailed derivation of \eqref{deltaExp}. We stress that the choice of $\kappa_0$ eliminates the contribution of the fourth-order coefficients $a_{4,0}$ and $c_{4, 0}$ to the leading term in \eqref{deltaExp}.  Choosing $\theta$ such that
\bq\label{cd:ttv2}
 |\tt|< \kappa_1, \quad \textrm{where} \quad \kappa_1 := \frac{2|b_{3,0}|}{|a_{0, 1}-c_{0, 1}|},
\eq
we appear to achieve $\Delta(\varepsilon,\delta) > 0$ for all sufficiently small $\varepsilon > 0$, as desired. This choice assumes, however, that $b_{3,0}$ does not vanish for all depths $h$.  Lemma ~\ref{lemm:b30v2} proves this is not the case, unless $h = h_{crit} =0.25065...$. Thus, provided $h \neq h_{crit}$, we have  $\Delta(\varepsilon,\delta) > 0$ for all sufficiently small  $\varepsilon > 0$.

Returning to the expression for the eigenvalues \eqref{ld:ABCDelta}, we find that $\text{Re}\lambda_+=\mez \sqrt{\Delta(\eps, \delta)} > 0$ for $h\ne h_{crit}$. 
Therefore, small-amplitude Stokes waves are unstable with respect to transverse perturbations with wave numbers $\alpha = \sqrt{\beta_*+\delta}$ for any  $\delta \in (\varepsilon^2\kappa_0-\varepsilon^3\ka_1, \varepsilon^2\kappa_0 +\varepsilon^3\ka_1)$, provided $\varepsilon$ is sufficiently small. Finally, by \eqref{deltaExp}, we have $\text{Re}\lambda_+ = \cO(\varepsilon^3)$ as $\varepsilon \rightarrow 0$, completing the proof.  \qed

Substituting \eqref{deltaExp} into \eqref{ld:ABCDelta} yields 
\bq\label{finalexpansion:ldpm}
\begin{aligned}
\ld_\pm\equiv \ld_\pm(\eps, \tt)%&= i\sigma +\frac{i}2(A+C) \pm \frac12\sqrt{\Delta(\eps,\delta)}\\
&= i\left[\sigma+\frac{a_{0,1}c_{2,0}-a_{2,0}c_{0,1}}{(a_{0,1}-c_{0,1})}\varepsilon^2 +\frac{a_{0, 1}+c_{0, 1}}{2}\tt\eps^3\right]\\
&\qquad \pm\mez \left[4b_{3, 0}^2-(a_{0, 1}-c_{0, 1})^2\tt^2\right]^\mez|\eps|^3+O(\eps^4).
\end{aligned}
\eq 
By eliminating $\tt$ in from the real and imaginary parts of $\ld_\pm$, we obtain 
 \begin{coro}\label{cor:ellipse}
 For sufficiently small $\eps$, the curve $(-\ka_1, \ka_1)\ni \tt\mapsto \ld_+(\eps, \tt)$ (resp. $(-\ka_1, \ka_1)\ni \tt\mapsto \ld_-(\eps, \tt)$) is within $O(\eps^4)$  distance to the entire right (resp. left) half of the ellipse 
% $\ld_\pm(\eps)$ approximately lie on the ellipse
\begin{align}\label{form:ellipse}
\frac{\lambda_r^2}{\left(b_{3,0}\varepsilon^3\right)^2} + \frac{\left(\lambda_i -\sigma - \left(\frac{a_{0,1}c_{2,0}-a_{2,0}c_{0,1}}{(a_{0,1}-c_{0,1})} \right) \varepsilon^2\right)^2}{\left(\frac{b_{3,0}(a_{0,1}+c_{0,1})}{a_{0,1}-c_{0,1}}\varepsilon^3\right)^2} =1
\end{align}
in the complex plane.   See Figure 1.  
\end{coro} 
%\begin{rema}\label{remark:ellipse}
%\begin{proof}
%The proof is identical to that in the infinite depth case \cite{CreNguStr}. 
%\end{proof}

%%%%%%%%%%%%%%%%%%%%%%%%%%%%%%%%%%%%%%%%%%%
%\appendix 
%\section{Normal trace of vector fields}
%We recall the following classical trace operator.
%\begin{prop}[\protect{Section  3.2, Chapter IV, \cite{BoyerFabrie}}]\label{prop:trace}
%Let $U\subset \Rr^n$ be Lipschitz domain with compact boundary and denote 
%\bq\label{def:Hdiv}
%H_{\cnx}(U)=\{u\in L^2(U; \Rr^n): \cnx u\in L^2(U)\}.
%\eq
%There exists a bounded linear operator $\gamma: H_{\cnx}(U)\to H^{-\mez}(\p U)$ such that $\gamma(u)=u\cdot \nu$ for $u\in C^\infty_0(\overline{U})$, where $\nu$ is the unit outward normal to $\p U$. In particular, there exists a constant $C>0$ such that 
%\bq\label{cont:tracegamma}
%\| \gamma(u)\|_{H^{-\mez}(\p U)}\le C\Big(\| u\|_{L^2(U)}+\| \cnx u\|_{L^2(U)}\Big)\quad\forall u\in H_{\cnx}(U).
%\eq
% Moreover, if $u\in H_{\cnx}(U)$ and $w\in H^1(U)$, then the Stokes formula 
%\bq\label{Stokes}
%\langle \gamma(u), \gamma_0(w)\rangle_{H^{-\mez}(\p U), H^\mez(\p U)}=\int_U u\cdot \na w+\int_U w\cnx u
%\eq
%holds, where $\gamma_0(w)$ is the trace of $w$.% The trace operator 
%\end{prop}
%%%%%%%%%%%%%%%%%%%%%%%%%%%
\vspace{.1in}
{\noindent{\bf{Acknowledgment.}}   The work of HQN was partially supported by NSF grant DMS-2205710. The work of RPC was partially supported by NSF grant DMS-2402044. 
}

%%%%%%%%

\end{document}